\documentclass[a4paper,12pt]{article}
\usepackage{amsmath,amsthm,amsfonts,amssymb,bbm}
\usepackage{graphicx,psfrag,subfigure}
\usepackage{cite}
\usepackage{hyperref}
\usepackage[hmargin=1in,vmargin=1in]{geometry}
\hypersetup{colorlinks=true}

\renewcommand{\d}{\mathrm d}

\newcommand{\ext}{\mathrm{ext}}
\newcommand{\R}{\mathbb R}
\newcommand{\Z}{\mathbb Z}
\newcommand{\wt}{\widetilde}
\newcommand{\wh}{\widehat}
\newcommand{\ol}{\overline}
\newcommand{\ul}{\underline}
\renewcommand{\Re}{\operatorname{Re}}
\renewcommand{\Im}{\operatorname{Im}}
\newcommand{\Ai}{\operatorname{Ai}}

\newcommand{\Var}{\operatorname{Var}}

\newcommand{\id}{\mathbbm{1}}
\renewcommand{\O}{\mathcal{O}}
\renewcommand{\P}{\mathbf P}
\newcommand{\E}{\mathbf E}

\newtheorem{proposition}{Proposition}[section]
\newtheorem{theorem}[proposition]{Theorem}
\newtheorem{lemma}[proposition]{Lemma}

\theoremstyle{definition}
\newtheorem{remark}[proposition]{Remark}

\numberwithin{equation}{section}

\author{Patrik L.\ Ferrari\thanks{Institute for Applied Mathematics, Bonn University, Endenicher Allee 60, 53115 Bonn,
Germany. E-mail: {\tt ferrari@uni-bonn.de}} \and
B\'alint Vet\H o\thanks{Department of Stochastics, Budapest University of Technology and Economics;
MTA\,--\,BME Stochastics Research Group, Egry J.\ u.\ 1, 1111 Budapest, Hungary. E-mail: {\tt vetob@math.bme.hu}}}

\title{Fluctuations of the arctic curve in the tilings of the Aztec diamond on restricted domains}

\date{}

\begin{document}

\maketitle
\begin{abstract}
We consider uniform random domino tilings of the restricted Aztec diamond which is obtained by cutting off an upper triangular part of the Aztec diamond by a horizontal line.
The restriction line asymptotically touches the arctic circle that is the limit shape of the north polar region in the unrestricted model.
We prove that the rescaled boundary of the north polar region in the restricted domain converges to the Airy$_2$ process conditioned to stay below a parabola
with explicit continuous statistics and the finite dimensional distribution kernels.
The limit is the hard-edge tacnode process which was first discovered in the framework of non-intersecting Brownian bridges.
The proof relies on a random walk representation of the correlation kernel of the non-intersecting line ensemble which corresponds to a orandom tiling.

\end{abstract}
\sloppy

\section{Introduction and main results}
The Aztec diamond is one of the best studied random tiling models.
It has been introduced in~\cite{EKLP92} and has been analyzed in great detail since then using different techniques.
A disordered region is located in the center of the Aztec diamond and there are four ordered ones at the corners where the tiling follows a completely regular pattern.
\begin{figure}
\begin{center}
\vspace{-1.5cm}\includegraphics[height=6cm,angle=-45]{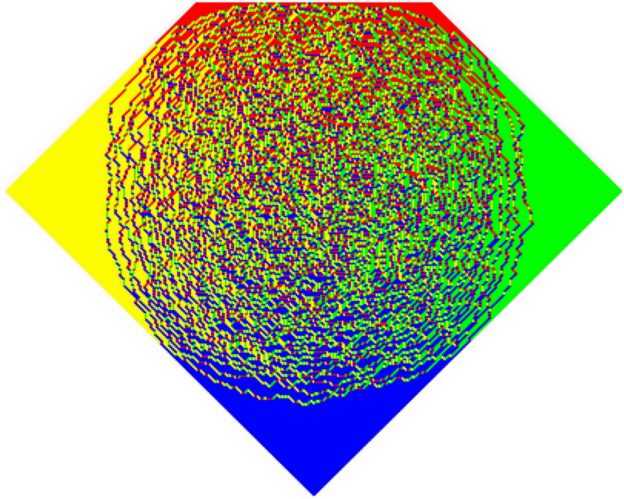}
\caption{One random realization of the Aztec diamond of size $n=100$ with restricted domain with parameter $R=2$ in \eqref{CutThreshold}. Code courtesy of Sunil Chhita.}
\label{FigAztec}
\end{center}
\end{figure}
The law of large number for the boundary of the disordered region, also known as arctic circle theorem, was shown in~\cite{JPS98},
while the limiting density of dominoes with a given orientation in the disordered region was obtained in~\cite{CEP96,CKP01}.
Using the inverse Kasteleyn matrix approach~\cite{Kas63}, one can analyze the analogue of the arctic circle also for more general domains and tiling models~\cite{KOS03}
as well as local statistics of the local field~\cite{Ken97,Ken01,Hel00,CJY15}.
The Aztec diamond model is known to be equivalent to the six-vertex model with domain wall boundary conditions at the free fermion point~\cite{ZJ00,AR05,FS06,BF05,CPS16b}.
Furthermore, the Aztec diamond can be generated by the shuffling algorithm~\cite{EKLP92} and a Markov chain of a system of interlaced particles system~\cite{Nor08,BF15}
and the boundary of the north polar region evolves as the discrete time TASEP with step initial condition.
This property was used to obtain the limit shape~\cite{JPS98} and it provides the link to the KPZ universality class of growth models.

The boundary of the disordered region can be studied by using a non-intersecting line ensemble,
for which the top line is exactly at the border of the disordered region, see Figure~\ref{FigLines}.
By the Lindstr\"om--Gessel--Viennot method, a discrete version of the Karlin--McGregor formula~\cite{KM59}, the lines form determinantal point process (see the book chapter~\cite{Bor10}).
In particular, the joint distributions of the top line at different times are given by a Fredholm determinant.
Using this technique~\cite{Jo03b}, it was shown in~\cite{Jo03} that the top line converges to the Airy$_2$ process~\cite{PS02}.
The line ensemble for the Aztec diamond fits into the class of Schur processes~\cite{Ok01} and it can be dynamically generated~\cite{BF15} as a consequence of the shuffling algorithm~\cite{EKLP92}.

In this paper we consider uniform tilings of the Aztec diamond in a restricted domain, see Figure~\ref{FigAztec}, which can be generated by a generalized shuffling algorithm~\cite{Pro03}.
More precisely, we cut off the top part of the Aztec diamond at a level which is in the natural fluctuation scale of the top line.
Equivalently, we can think of conditioning the random tiling to be ordered above the line of restriction, or
in terms of the corresponding non-intersecting line ensemble, it is equivalent for the lines to stay below a fixed threshold.
Our main result is the convergence of the top line to the so-called hard-edge tacnode process $\cal T$, which has been identified as the limit of non-intersecting Brownian bridges in~\cite{FV16}.

The hard-edge tacnode process in the context of non-intersecting lines was first described in~\cite{D13}:
non-intersecting squared Bessel processes were investigated and the one-point marginal distributions of their scaling limit at the hard-edge tacnode
were identified in terms of the solution of a $4\times4$ Riemann--Hilbert problem.
Fredholm determinant formulas with explicit kernels were obtained later in~\cite{DV14} for the multi-time distribution of the same process provided that the dimension of the Bessel paths is an even integer.
It does not include the case of non-intersecting Brownian bridges which were studied in~\cite{FV16}.
Shortly afterwards, the finite dimensional distributions of non-intersecting Brownian bridges in the limit close to the hard-edge tacnode
were described in~\cite{LW17} in a different formulation involving special functions related to the Painlev\'e II equation.
In addition, in was proved in~\cite{LW17} that the hard-edge tacnode kernel of~\cite{FV16} is the odd part of the soft-edge tacnode kernel of~\cite{FV11}.

As a byproduct along the proof of our main results in this paper,
we derive a Fredholm determinant formula in Theorem~\ref{thm:ReformulationKernel} for the continuum statistics of the Airy$_2$ process ${\cal A}_2$
in terms of the hitting time and position of Brownian motion in the spirit of~\cite{QR19}.
Analogous formulas coming from a different approach can be found in the KPZ fixed point paper~\cite{MQR17} (see Propositions~3.6, 3.8 and 4.4 therein).

\subsubsection*{The Aztec diamond model}
We follow the notations of~\cite{Jo03}.
The Aztec diamond is a domain $A_n$ in $\R^2$ that consists of the union of squares of the form $[k,k+1]\times[l,l+1]$ which lie inside $\{|x|+|y|\le n+1\}$.
In the original problem, one of all possible tilings of $A_n$ by vertical or horizontal $2\times1$ dominos is chosen uniformly at random.

Let us introduce a coloring of the squares in the Aztec diamond in a checkerboard fashion in a way that in the top half of $A_n$, the leftmost square of each row is white.
We call a horizontal domino a north domino if its leftmost square is white, otherwise call it a south domino.
Similarly, a vertical domino is a west domino if its upper square is white and it is an east domino otherwise.
The north polar region is the connected component of all north dominoes adjacent to the boundary of $A_n$.
Similarly, south, west and east polar regions can be defined.

In order to study the fluctuations of the boundary of the north polar region around its asymptotic shape, in~\cite{Jo02b}
each tiling configuration of the Aztec diamond was mapped into a system of non-intersecting lines as follows, see Figure~\ref{FigLines}.
On each south domino which has corners at $(0,0)$ and at $(2,1)$, a line is drawn from $(0,1/2)$ to $(2,1/2)$, on north dominoes no lines are drawn.
On each west domino which has corners at $(0,0)$ and at $(1,2)$, a line from $(0,1/2)$ to $(1,3/2)$ is drawn, similarly a line from $(0,3/2)$ to $(1,1/2)$ is drawn on each east domino.
Let $X_n(t)$ denote the top curve of the line ensemble from $(-n,-1/2)$ to $(n,-1/2)$, which follows the boundary of the north polar region.
\begin{figure}
\begin{center}
\vspace{-1.5cm}\includegraphics[height=6cm,angle=-45]{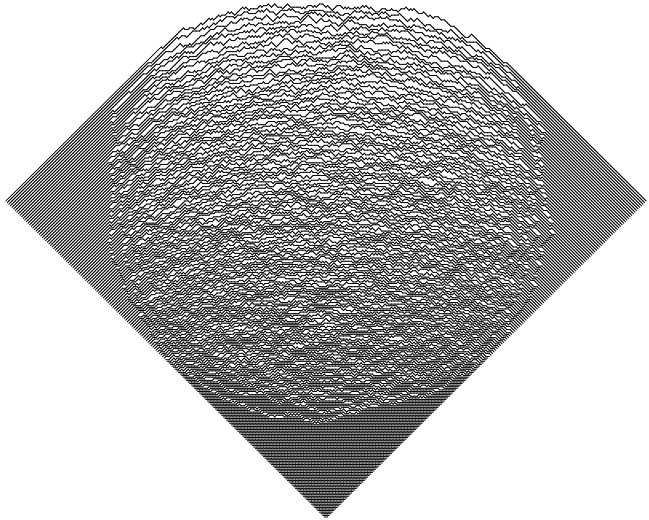}
\caption{Lines associated with the Aztec diamond of Figure~\ref{FigAztec}. Code courtesy of Sunil Chhita.}
\label{FigLines}
\end{center}
\end{figure}

\subsubsection*{The main results}
In~\cite{Jo03} it is shown that the boundary of the north polar region of the Aztec diamond has Airy$_2$ fluctuations on the $n^{1/3}$ scale with respect to the limit shape.
Thus, to obtain a non-trivial interaction on the fluctuations scale with the limiting Airy$_2$ process on the boundary of the north polar region,
we consider a uniform tiling of the Aztec diamond $A_n$ restricted to $y\le r$ where the horizontal line $y=r$ is set to be on the $n^{1/3}$ scale around the top of the limit shape.
We prove that the boundary of the north polar region in the restricted model converges to the hard-edge tacnode process in terms of continuum statistics and finite dimensional distributions,
see Theorems~\ref{thm:Xbelowh} and~\ref{thm:Xfindimdistr}.

The hard-edge tacnode process was described first in~\cite{FV16} as the $n\to\infty$ limit of $n$ non-intersecting Brownian bridges conditioned to stay below a fixed threshold.
The limiting correlation kernel of the restricted Brownian bridges as well as the probability that it lies below a given function has been obtained in~\cite{FV16}.
The hard-edge tacnode $\cal T$ can also be described as the $u\mapsto {\cal A}_2(u)-u^2$ process conditioned to stay below a given threshold, thus it is a one-parameter family process.

In this paper, we cut off a triangle at the top corner of the Aztec diamond at height
\begin{equation}\label{CutThreshold}
r=n/\sqrt{2}+2^{-5/6}Rn^{1/3}
\end{equation}
for a given fixed $R\in\R$. Denote by $X_n^R$ the top line of the corresponding non-intersecting line ensemble in this case and introduce its rescaled position as
\begin{equation}\label{defXnRresc}
X_n^{R,{\rm resc}}(t)=\frac{X_n^R(2^{-1/6}tn^{2/3})-n/\sqrt{2}}{2^{-5/6}n^{1/3}}.
\end{equation}
Since the height of the threshold is scaled as the top line in the unrestricted model, the conditioning is relevant in the limit. To state the results, we first introduce the limiting kernel.

For any parameter $s\in\R$, let
\begin{equation}\label{eqAiryS}
\Ai^{(s)}(x)=e^{2s^3/3+xs}\Ai(s^2+x).
\end{equation}
Then we define the functions
\begin{equation}\label{defPhihat}\begin{aligned}
\Phi_t^\xi(u)&=\Ai^{(t)}(R+\xi+u)-\Ai^{(t)}(R+\xi-u),\\
\Psi_t^\zeta(u)&=\Ai^{(-t)}(R+\zeta+u)-\Ai^{(-t)}(R+\zeta-u),
\end{aligned}\end{equation}
and the shifted GOE kernel
\begin{equation}
K_0(\xi,\zeta)=2^{-1/3}\Ai(2^{-1/3}(2R+\xi+\zeta)).\label{defK0hat}
\end{equation}
For a given function $g:\R\to\R$, let us define the following transition density
\begin{equation}\label{defTg}
T^g_{t_1,t_2}(\xi,\zeta)=\frac\partial{\partial\zeta}\P_{b(t_1)=\xi}(b(t)\leq g(t)-t^2,t\in[t_1,t_2],b(t_2)\le\zeta)
\end{equation}
where $b(t)$ is a Brownian motion with diffusion coefficient $2$.
Then the limiting kernel is given by
\begin{equation}\label{defKext}
K^\ext(t_1,u_1;t_2,u_2)=-\id_{t_1<t_2}T^0_{t_1,t_2}(u_1,u_2)+\int_{\R_+}\d\xi\int_{\R_+}\d\zeta\,\Psi_{t_1}^\xi(u_1)(\id-K_0)^{-1}(\xi,\zeta)\Phi_{t_2}^\zeta(u_2)
\end{equation}
where $t_1,t_2\in\R$ and $u_1,u_2\le0$.
We remark that above and in the rest of the paper we use the same notation for an integral operator and for its kernel.

Let us state the main results of the paper about the limiting distribution of the rescaled top line in a random tiling of the Aztec diamond on a restricted domain.

\begin{theorem}\label{thm:Xbelowh}
For $t_1<t_2$, let $g:\R\to\R$ be a function with $g(t)=R+t^2$ for $t\not\in [t_1,t_2]$ and $g(t)\le R+t^2$ on $[t_1,t_2]$.
Suppose that $g$ is differentiable on $[t_1,t_2]$ except for countably many points where it may be discontinuous
and assume that its derivative is square integrable on the intervals on which $g$ is differentiable.
Then
\begin{equation}\label{Xbelowh}\begin{aligned}
\lim_{n\to\infty} \P(X_n^{R,{\rm resc}}(t)\le g(t)-t^2, t\in[t_1,t_2])& = \frac{\P({\cal A}_2(t)\le g(t)\textrm{ for all }t\in\R)}{\P({\cal A}_2(t)\leq R+t^2\textrm{ for all }t\in\R)}\\
&=\det(\id-K_{t_1,t_1}+T^{g-R}_{t_1,t_2}K_{t_2,t_1})_{L^2(\R_-)}
\end{aligned}\end{equation}
where $K_{t_2,t_1}(u,v)=K^\ext(t_2,u,t_1,v)$ given in \eqref{defKext} and the right-hand side above is the same as the right-hand side of (2.41) in~\cite{FV16}.
\end{theorem}

\begin{theorem}\label{thm:Xfindimdistr}
For arbitrary $t_1,\dots,t_k$ and $u_1,\dots,u_k\le R$, we have
\begin{equation}\label{Xfindimdistr}
\lim_{n\to\infty} \P\left(\bigcap_{\ell=1}^k \{X_n^{R,{\rm resc}}(t_\ell)\le u_\ell\}\right) = \det\left(\id-K^\ext\right)_{L^2(E)}
\end{equation}
with the set $E=\{(t_1,[u_1-R,0])\times\dots\times (t_k,[u_k-R,0])\}$.
\end{theorem}

\subsubsection*{The method}
To obtain our results, we consider the probabilities on the left-hand side \eqref{Xbelowh} and \eqref{Xfindimdistr} as conditional probabilities.
Then we determine the limit of the probability that $X_n(t)$ remains below a given function (see Theorem~\ref{thm:YbelowacurveR})
for which the probability of the condition is just the function $g(\tau)=R+\tau^2$.

The first step is to map the boundary curve of the north polar region for the Aztec diamond $X_n(t)$
into the top curve of the non-intersecting lines $Y_n(t)$ from $Y_n(0)=0$ to $Y_n(2n)=0$ following~\cite{Jo03}.
The possible steps of the lines in the ensemble alternate for odd and even steps.
The lines can stay or increase by one in odd steps, in particular $Y_n(2j+1)$ can be equal to either $Y_n(2j)$ or $Y_n(2j)+1$ for $j=0,1,\dots,n-1$.
The lines can stay or decrease by any positive integer in even steps, in particular $Y_n(2j+2)\le Y_n(2j+1)$ for $j=0,1,\dots,n-1$.
This line ensemble representation forms a Schur process.

The bijection which maps the line $X_n$ into $Y_n$ is described in~\cite{Jo03} in detail
and it has the property that if $(t,x-1/2)$ is a point on the curve $X_n$ where $t\in[-n,n]$ and $x\in[0,n]$ are integers, then the point $(t+x+n,x)$ is on the curve $Y_n$.
By neglecting the integer parts we could conclude that the event $\{X_n(t)\le x\}$ and $\{Y_n(t+x+n)\le x\}$ are equal.
As suggested by the rescaling \eqref{defXnRresc}, we choose $t=2^{-1/6}\tau n^{2/3}$ and $x=n/\sqrt2+2^{-5/6}\xi n^{1/3}$ as parameters of the events above.
We introduce the time scaling and the rescaled version of any function $g:\R\to\R$ as
\begin{equation}\label{defbtaugn}
b_n(\tau)=n\left(1+\frac1{\sqrt2}\right)+2^{-1/6}\tau n^{2/3},\qquad g_n(\tau)=\frac n{\sqrt 2}+2^{-5/6}(g(\tau)-\tau^2)n^{1/3}.
\end{equation}
Considering the event $\{Y_n(t+x+n)\le x\}$, the argument of the process $Y_n$ is equal to $t+x+n=b_n(\tau)+2^{-5/6}\xi n^{1/3}$ by the scaling above.
The $2^{-5/6}\xi n^{1/3}$ is negligible in the $n\to\infty$ limit compared to $b_n(\tau)$.
For this reason we consider events of the form $\{Y_n(b_n(\tau))\le g_n(\tau), \tau\in[t_1,t_2]\}$ which is asymptotically the same as the probability on the left-hand side of \eqref{Xbelowh}.
In particular, the condition that $X_n$ stays below a given threshold maps to the condition that $Y_n$ stays below the same threshold.

Theorem~\ref{thm:YbelowacurveR} about the convergence of probabilities above
carries the most important input for the proof of Theorems~\ref{thm:Xbelowh} and~\ref{thm:Xfindimdistr} on the top curve of the original ensemble.
In order to state it, let $H^1_\ext(\R)$ be the class of functions $g$ for which there is a finite or infinite collection of intervals $[a,b]$ with $a\le b$
so that $g$ is differentiable on $[a,b]$ with derivative in $L^2([a,b])$ and $g=\infty$ outside this collection of intervals.

\begin{theorem}\label{thm:YbelowacurveR}
Let $g\in H^1_\ext(\R)$.
Then
\begin{equation}\label{YbelowacurveR}
\P\left(Y_n(b_n(\tau))\le g_n(\tau)\mbox{ for }\tau\in\R\right)\to\P\left(\mathcal A_2(\tau)\le g(\tau)\mbox{ for }\tau\in\R\right)
\end{equation}
as $n\to\infty$ where $g_n$ is given by \eqref{defbtaugn} and $\mathcal A_2$ is the Airy$_2$ process.
\end{theorem}

In the rest of this section we provide the basic ingredients for the proof of Theorem~\ref{thm:YbelowacurveR}.
The finite dimensional distributions of the curve $Y_n(t)$ which appears on the left-hand side of \eqref{YbelowacurveR} is expressed in terms of Fredholm determinants in~\cite{Jo03} as follows.
Let
\begin{equation}\label{defT}
T(x,y)=\left\{\begin{array}{cl} 2 & \mbox{if }y\le x\\ 1 & \mbox{if }y=x+1\\ 0 & \mbox{if }y>x+1\end{array}\right.
\end{equation}
be the transition function of the non-intersecting line ensemble for even times.
It will be sufficient for our purposes to consider even times only.
Define the correlation kernel
\begin{equation}\label{defKn}
K_n(2r,x;2s,y)=-\id_{r<s}T^{s-r}(x,y)+\wt K_n(2r,x;2s,y)
\end{equation}
with
\begin{equation}\label{defKntilde}
\wt K_n(2r,x;2s,y)=\frac{-1}{(2\pi i)^2}\oint_{\Gamma_1}\frac{\d z}z\oint_{\Gamma_0}\frac{\d w}w \frac{w^y(1-w)^{n-s}(1+1/w)^s}{z^x(1-z)^{n-r}(1+1/z)^r}\frac z{z-w}
\end{equation}
where, for a set $J$, the integration contour $\Gamma_J$ is a simple counter-clockwise loop including only the poles at $J$. Thus $\Gamma_0$ goes around $0$ and $\Gamma_1$ around $1$ without intersecting.
Note that in~\cite{Jo03}, the contour $\Gamma_1$ was a vertical line in the kernel \eqref{defKntilde},
but it can be deformed to a circle around $1$ as long as the integrand has no singularity at infinity, i.e.\ $x\ge-n+r$.
The latter condition physically means that we choose a space-time location corresponding to the Aztec diamond.

By writing $z/(z-w)=\sum_{j\ge0}(w/z)^j$, the kernel becomes
\begin{equation}\label{kernelrewrite}
\wt K_n(2r,x;2s,y)=\sum_{j\ge0} p_r^{(n)}(x+j)q_s^{(n)}(y+j)
\end{equation}
where the functions
\begin{align}
p_r^{(n)}(x)&=\frac{-1}{2\pi i}\oint_{\Gamma_1}\frac{\d z}z\frac1{z^x(1-z)^{n-r}(1+1/z)^r},\label{defp}\\
q_s^{(n)}(y)&=\frac1{2\pi i}\oint_{\Gamma_0}\frac{\d w}ww^y(1-w)^{n-s}(1+1/w)^s\label{defq}
\end{align}
are related to Krawtchouk polynomials.
The distribution of the non-intersecting line ensemble is characterized as follows.

\begin{proposition}[See (2.22) of~\cite{Jo03}]\label{prop:corrkernel}
Let $2t_0<2t_1<\dots<2t_k$ be even integers in $\{2,\ldots,2n-2\}$ and $x_0,x_1,\dots,x_k$ be integers.
Let $E_i=\{(2t_i,y_i):y_i\in\Z,y_i > x_i\}$ for $i=0,1,\dots,k$ and $E=\cup_{i=0}^k E_i$.
Then
\begin{equation}\label{Yfinitedimdistr}
\P\left(\cap_{i=0}^k \{Y_n(2t_i)\le x_i\}\right)=\det\left(\id-K_n\right)_{\ell^2(E)}
\end{equation}
where $K_n$ is the correlation kernel of the point process and it is given by \eqref{defKn}.
\end{proposition}

As part of the proof of Theorem~\ref{thm:YbelowacurveR}
the right-hand side of \eqref{YbelowacurveR} with $\tau\in\R$ replaced by a finite interval $\tau\in[L,M]$ can be written as a Fredholm determinant using Theorem 2 of~\cite{CQR11} which we state below.
The theorem was stated in~\cite{CQR11} only for function $g\in H^1([L,M])$, i.e.\ with derivative in $L^2([L,M])$,
but the proof in~\cite{CQR11} applies for $H^1_\ext([L,M])$, see also subsequent works~\cite{QR19,MQR17}.
We denote by
\begin{equation}
K_{\Ai}(x,y)=\int_0^\infty\d\lambda\Ai(x+\lambda)\Ai(y+\lambda)
\end{equation}
the Airy kernel and by $H=-\partial_x^2+x$ the Airy Hamiltonian.
\begin{theorem}[Theorem 2 and 3 of~\cite{CQR11}]\label{thm:CQR}
Let $L<M$ be fixed and $g\in H^1_\ext([L,M])$.
Then
\begin{equation}\label{CQR}
\P\left(\mathcal A_2(\tau)\le g(\tau)\mbox{ for }\tau\in[L,M]\right)=\det\left(\id-\left(e^{(L-M)H}-\Theta_{L,M}^g\right)e^{(M-L)H}K_{\Ai}\right)_{L^2(\R)}
\end{equation}
where one can write\footnote{We have corrected a typo in Theorem~3 of~\cite{CQR11} in the Gaussian prefactor on the r.h.s.~of \eqref{CQRkernel}.}
\begin{multline}\label{CQRkernel}
\left(e^{(L-M)H}-\Theta_{L,M}^g\right)(u,v)=e^{Lu-Mv-L^3/3+M^3/3}\frac{e^{-\frac{(u-L^2-v+M^2)^2}{4(M-L)}}}{\sqrt{4\pi(M-L)}}\\
\times\left(1-\P_{b(L)=u-L^2,b(M)=v-M^2}\left(b(\tau)\le g(\tau)-\tau^2\mbox{ for }\tau\in[L,M]\right)\right)
\end{multline}
with $b(\tau)$ being a Brownian bridge from $u-L^2$ to $v-M^2$ with diffusion coefficient $2$.
\end{theorem}

Finally we state an alternative Fredholm determinant expression for the probability that the Airy$_2$ process stays below a given function on a fixed interval.
The kernel involves the hitting time and position of a Brownian motion and it serves as an important ingredient for the proof of Theorem~\ref{thm:Xbelowh}.
On the other hand, Theorem~\ref{thm:ReformulationKernel} is of independent interest.

Let us consider $\alpha\in[L,M]$.
The Brownian bridge $b$ that appears on the right-hand side of \eqref{CQRkernel} can be thought of as the concatenation of a bridge $b(L)=u-L^2$ to $b(\alpha)=\xi$
and another one from $b(\alpha)=\xi$ to $b(M)=v-M^2$ where $\xi$ is an arbitrary value.
Equivalently, given $b(\alpha)=\xi$, one can think of the first bridge starting from $b(\alpha)=\xi$ and going backwards in time until $b(L)=u-L^2$.

Let us introduce the following hitting times and positions
\begin{align}
T_+^{\xi,\alpha}& = \inf\{t>\alpha | b(t)\not\leq g(t)-t^2\textrm{ with }b(\alpha)=\xi\},& X_+^{\xi,\alpha}&= b(T_+^{\xi,\alpha}),\label{defTX+}\\
T_-^{\xi,\alpha}& = \sup\{t<\alpha | b(t)\not\leq g(t)-t^2\textrm{ with }b(\alpha)=\xi\},& X_-^{\xi,\alpha}&= b(T_-^{\xi,\alpha}).\label{defTX-}
\end{align}
Recall the definition of $\Ai^{(s)}$ given in \eqref{eqAiryS} and let
\begin{equation}\label{defM}\begin{aligned}
M_-^\alpha(x,\xi)&=\int_{L}^\alpha \int_\R \P(T_-^{\xi,\alpha}\in\d t, X_-^{\xi,\alpha}\in d\zeta) \Ai^{(t)}(\zeta+x),\\
M_+^\alpha(\xi,y)&=\int_{\alpha}^M \int_\R \P(T_+^{\xi,\alpha}\in\d t, X_+^{\xi,\alpha}\in\d\zeta) \Ai^{(-t)}(\zeta+y).
\end{aligned}\end{equation}
Remark that if $\xi\geq g(\alpha)-\alpha^2$, then $\P(T_+^{\xi,\alpha}\in\d t, X_+^{\xi,\alpha}\in\d\zeta)=\delta_\alpha(t)\delta_\xi(\zeta)$,
from which $M_+^\alpha(\xi,y)=\Ai^{(-\alpha)}(\xi+y)$ and similarly $M_-^\alpha(x,\xi)=\Ai^{(\alpha)}(\xi+x)$.

\begin{theorem}\label{thm:ReformulationKernel}
Let $g\in H^1_\ext([L,M])$.
We have
\begin{equation}\label{ReformulationKernel}
\P({\cal A}_2(r)\leq g(r), r\in [L,M]) = \det(\id-K)_{L^2(\R_+)}
\end{equation}
with kernel given by
\begin{multline}\label{defKreformulated}
K(x,y)=\int_\R\d\xi M_-^\alpha(x,\xi) \Ai^{(-\alpha)}(\xi+y)\\+ \int_\R\d\xi \Ai^{(\alpha)}(\xi+x)M_+^\alpha(\xi,y) -\int_\R\d\xi M_-^\alpha(x,\xi) M_+^\alpha(\xi,y)
\end{multline}
where $\alpha\in[L,M]$ is arbitrary.
\end{theorem}

\begin{remark}\label{remark3.2}
The idea of the decomposition goes back to~\cite{QR19}.
However the formula stated there holds true only when the hitting position of the Brownian bridge is on the graph of the function $g(t)-t^2$.
As in~\cite{MQR17}, the present formulas include the possibility that the Brownian motion hits at a position $t$ strictly greater than $g(t)-t^2$, which can happen if $g$ is not continuous.

For finite $L,M$, the kernel is well-defined if $g$ is bounded from below, while in the $L\to-\infty$ and/or $M\to\infty$,
it should be enough to consider $g(t)\geq c+\epsilon t^2$ for some $\epsilon>0$, as in this case the probability that the Airy$_2$ process stays below $g$ at all times remains strictly positive.
In the representation of~\cite{QR19}, it is shown to hold at least whenever $g(t)\geq c+\epsilon t^2$ for some $\epsilon>1/4$.
\end{remark}

Theorem~\ref{thm:ReformulationKernel} is used in the proof of our main result in the special case when $\alpha=L$ is the left endpoint of the interval.
Then the kernel in \eqref{defKreformulated} is even simpler, see Theorem~\ref{thm:belowcurverewrite} below for the explicit expression and for its discrete counterpart.
On the other hand, the choice of the left endpoint avoids certain difficulties with the backward part of the random walk in the asymptotic analysis.

\subsubsection*{Outline}
The rest of the paper is organized as follows.
In Section~\ref{s:rw}, we provide a random walk representation of probabilities that appear on the left-hand side of \eqref{YbelowacurveR} in the finite time case
as a discrete analogue of Theorem~\ref{thm:CQR}.
We prove Theorem~\ref{thm:ReformulationKernel} about the continuum statistics for the Airy$_2$ process in terms of hitting times and positions in Section~\ref{s:hitting}.
The setup is specialized for the case when we start the Brownian motion at the left endpoint of the interval in Section~\ref{s:leftepa}
where also the analogous Fredholm determinant formulas are given in the discrete case for the hitting times and positions of the corresponding random walk.
Section~\ref{s:asympstate} contains the asymptotic statements of the paper in the $n\to\infty$ limit
which together lead to Theorem~\ref{thm:Ybelowacurve} that is the finite interval version of Theorem~\ref{thm:YbelowacurveR}.
Section~\ref{s:extension} extends Theorem~\ref{thm:Ybelowacurve} to the full line statement and proves Theorem~\ref{thm:YbelowacurveR}.
The main results of the paper, Theorems~\ref{thm:Xbelowh} and~\ref{thm:Xfindimdistr} are derived from Theorem~\ref{thm:YbelowacurveR} in Section~\ref{s:identification}.
The proofs of the asymptotic statements are postponed to Section~\ref{s:asympproofs}.

\paragraph{Acknowledgements:}
This work has originated in discussions with F.\ Colomo and A.\ Sportiello about their work and to both ICERM and the Galileo Galilei Institute,
which provided the platform to make such discussions possible.
The work of P.L.~Ferrari is supported by the German Research Foundation through the Collaborative Research Center 1060 ``The Mathematics of Emergent Effects'', project B04,
and by the Deutsche Forschungsgemeinschaft (DFG, German Research Foundation) under Germany's Excellence Strategy -- GZ 2047/1, Projekt ID 390685813.
The work of B.\ Vet\H o was supported by the NKFI (National Research, Development and Innovation Office) grants PD123994 and FK123962
and by the Bolyai Research Scholarship of the Hungarian Academy of Sciences.

\section{Random walk representation}\label{s:rw}

The joint distribution of the top curve $Y_n$ of non-intersecting lines at different times is characterized with the extended correlation kernel $K_n$ in Proposition~\ref{prop:corrkernel}.
The main contribution of this section towards the proof of Theorem~\ref{thm:CQR} is that we rewrite \eqref{Yfinitedimdistr} in a Fredholm determinant form
where the kernel is a discrete analogue of \eqref{CQRkernel}, i.e.\ it involves a probability that a certain random walk remains below given values at various steps, see Proposition~\ref{prop:Yrw} below.
A random walk representation of similar spirit also appears in~\cite{MQR17}.

Let
\begin{equation}
T_1(x,y)=\left\{\begin{array}{cl} 1 & \mbox{if }y\in\{x,x+1\}\\ 0 & \mbox{otherwise}\end{array}\right.\quad\mbox{and}\quad
T_2(x,y)=\left\{\begin{array}{cl} 1 & \mbox{if }y\le x\\ 0 & \mbox{otherwise}\end{array}\right.
\end{equation}
be two transition operators.
Then $T_1$ and $T_2$ commute and by the construction of~\cite{Jo03}, the operator $T$ given in \eqref{defT} appears as $T=T_1T_2$.

\begin{proposition}\label{prop:rwconstruction}
Let $X_1$ and $X_2$ be two independent random variables with the distributions
\begin{equation}\label{defX}\begin{aligned}
\P(X_1=0)&=\frac{\sqrt2-1}{\sqrt2},\quad\P(X_1=1)=\frac1{\sqrt2},\\
\P(X_2=-k)&=\sqrt2(\sqrt2-1)^{k+1}\quad k=0,1,2,\dots
\end{aligned}\end{equation}
and let $X=X_1+X_2$.
Then $\E(X)=0$, $\Var(X)=\sqrt2$, and
\begin{equation}\label{Tkernel}
T(x,y)=(\sqrt2+1)^{2-y+x}\P(X=y-x).
\end{equation}
Equivalently, for any function $f\in\ell^1(\Z)$,
\begin{equation}\label{Trepr}
(Tf)(x)=(\sqrt2+1)^2\E\left((\sqrt2+1)^{-X}f(x+X)\right).
\end{equation}
\end{proposition}

\begin{proof}[Proof of Proposition~\ref{prop:rwconstruction}]
Slightly more generally, let
\begin{equation}\label{defrwsteps}\begin{aligned}
\P(X_1=0)&=\frac p{p+1},\quad\P(X_1=1)=\frac1{p+1},\\
\P(X_2=-k)&=(1-p)p^k\quad k=0,1,2,\dots
\end{aligned}\end{equation}
for some $p\in(0,1)$.
Then
\begin{align}
\E\Big(\big(\tfrac1p\big)^{-X_1}f(x+X_1)\Big)&=\frac p{p+1}f(x)+\frac1{p+1}\left(\frac1p\right)^{-1}f(x+1)=\frac p{p+1}(T_1f)(x),\\
\E\Big(\big(\tfrac1p\big)^{-X_2}f(x+X_2)\Big)&=\sum_{k=0}^\infty(1-p)p^k\left(\frac1p\right)^kf(x-k)=(1-p)(T_2f)(x).
\end{align}
Hence using $X=X_1+X_2$ and $T=T_1T_2$ we get
\begin{equation}\label{Tprepr}
\E\Big(\big(\tfrac1p\big)^{-X}f(x+X)\Big)=\frac{p(1-p)}{p+1}(Tf)(x).
\end{equation}
Simple computations yield
\begin{equation}\begin{aligned}
\E(X_1)&=\frac1{p+1},\quad\Var(X_1)=\frac1{p+1}-\frac1{(p+1)^2},\\
\E(X_2)&=-\frac p{1-p},\quad\Var(X_2)=\frac p{(1-p)^2}.
\end{aligned}\end{equation}
The condition $\E(X)=0$ is satisfied if $p=\sqrt2-1$.
In this case, $\Var(X)=\sqrt2$ and \eqref{Tprepr} reduces to \eqref{Trepr}.
\end{proof}

For the rest of the paper, we introduce the notation
\begin{equation}\label{rw}
S_m=X^{(1)}+X^{(2)}+\dots+X^{(m)}\qquad m=1,2,\dots
\end{equation}
for the random walk with step distribution given by the operator $T$
where the sequence $X^{(1)},X^{(2)},\dots$ of steps are independent and distributed as the random variable $X$ defined in Proposition~\ref{prop:rwconstruction}.
Next we write the probability on the the left-hand side of \eqref{Yfinitedimdistr} as a Fredholm determinant of a path integral kernel based on~\cite{BCR13}
and we also rewrite the path integral kernel in terms of the random walk $S_m$ given in \eqref{rw}.
This provides a discrete analogue of Theorem~\ref{thm:CQR}.

\begin{proposition}\label{prop:Yrw}
Let $0<2t_0<2t_1<\dots<2t_k<2n$ be even integers and $x_0,x_1,\dots,x_k$ be integers.
Then
\begin{multline}\label{Ypathint}
\P\left(\cap_{i=0}^k \{Y_n(2t_i)\le x_i\}\right)\\
=\det\left(\id-\wt K_n(2t_0,\cdot;2t_0,\cdot)+\ol P_{x_0}T^{t_1-t_0}\ol P_{x_1}\cdots T^{t_k-t_{k-1}}\ol P_{x_k}\wt K_n(2t_k,\cdot;2t_0,\cdot)\right)_{\ell^2(\Z)}
\end{multline}
holds where the projections are given by $\ol P_af(x)=\id_{x\le a}f(x)$. Furthermore, we have
\begin{multline}\label{pathintegralrw}
\ol P_{x_0}T^{t_1-t_0}\ol P_{x_1}\cdots T^{t_k-t_{k-1}}\ol P_{x_k}(x,y)\\
=(\sqrt2+1)^{2(t_k-t_0)+x-y}\P(S_{t_k-t_0}=y-x)\P_{S_0=0,S_{t_k-t_0}=y-x}(x+S_{t_i-t_0}\le x_i\mbox{ for }i=0,\dots,k)
\end{multline}
where $S_j$ is the random walk defined in \eqref{rw}.
\end{proposition}

\begin{proof}[Proof of Proposition~\ref{prop:Yrw}]
We apply Theorem 3.3 of~\cite{BCR13} on the space $X=\Z$ with the operators $\ol Q_{x_i}=\ol P_{x_i}$, $\mathcal W_{t_i,t_j}=T^{t_j-t_i}$ and with the kernel $K_{t_i}=\wt K_n(2t_i,\cdot;2t_i,\cdot)$.
What one needs to check is that
\begin{equation}
\mathcal W_{t_i,t_j}K_{t_j}=K_{t_i}\mathcal W_{t_i,t_j}=\wt K_n(2t_i,\cdot;2t_j,\cdot).
\end{equation}
This can be seen from the representation (2.14) in~\cite{Jo03} of the kernel $\wt K_n$, it follows directly from the semigroup property of the transitions.
The assumptions on finite trace class norm of certain operators that appear hold due to the fact that $\wt K_n$ is the kernel for a line ensemble confined to a finite region,
hence the kernel is supported on a finite set of space time points.
This proves \eqref{Ypathint}.

To show \eqref{pathintegralrw}, we apply Proposition~\ref{prop:rwconstruction} inductively for the kernel on the left-hand side and we obtain that for any $f\in\ell^1(\Z)$
\begin{equation}\begin{aligned}
&\ol P_{x_0}T^{t_1-t_0}\ol P_{x_1}\dots T^{t_k-t_{k-1}}\ol P_{x_k}f(x)\\
&\qquad=(\sqrt2+1)^{2(t_k-t_0)}\E\bigg((\sqrt2+1)^{-S_{t_k-t_0}}f(x+S_{t_k-t_0})\prod_{i=0}^k\id_{x+S_{t_i-t_0}\le x_i}\bigg)\\
&\qquad=(\sqrt2+1)^{2(t_k-t_0)}\sum_{y\in\Z}\P(S_{t_k-t_0}=y-x)(\sqrt2+1)^{-y+x}f(y)\\
&\qquad\qquad\times\P_{S_0=0,S_{t_k-t_0}=y-x}(x+S_{t_i-t_0}\le x_i\mbox{ for }i=0,\dots,k)
\end{aligned}\end{equation}
where we computed the expectation with respect to the endpoint of the random walk $S_{t_k-t_0}$ in the second equation.
\end{proof}

Next we prove large deviation bounds for the random walk $S_m$ defined in \eqref{rw}.
We introduce the rate function
\begin{multline}\label{defIx}
I(x)=(2-x)\log(\sqrt2+1)+(1-x)\log\bigg(\frac{1-x}{1+\sqrt{1+(1-x)^2}}\bigg)\\
+\log\bigg(\frac{x+\sqrt{1+(1-x)^2}}{2-x+\sqrt{1+(1-x)^2}}\bigg)
\end{multline}
for any $x<1$.

\begin{proposition}\label{prop:ldp}
Consider the random walk $S_m$ in \eqref{rw}.
Then for any $m>0$ integer and $x\in[0,1)$,
\begin{equation}\label{ldpsup}
\P\left(\sup_{0\le k\le m}S_k\ge xm\right)\le e^{-mI(x)}
\end{equation}
holds with the rate function $I(x)$ is given in \eqref{defIx}.
In particular,
\begin{equation}\label{ldpupper}
\P(S_m\ge xm)\le e^{-mI(x)}.
\end{equation}
Furthermore, there is an $\varepsilon>0$ for which
\begin{equation}\label{Ixsmallbound}
I(x)\ge\varepsilon x^2
\end{equation}
for any $x\ge0$.
As a consequence, the upper bounds on the right-hand sides of \eqref{ldpsup} and \eqref{ldpupper} can be replaced by $e^{-\varepsilon mx^2}$ if $x\ge0$.
\end{proposition}

\begin{proof}[Proof of Proposition~\ref{prop:ldp}]
Let $u>0$ be arbitrary.
Then the function $x\mapsto e^{ux}$ is increasing and convex, hence the process $e^{uS_m}$ is a non-negative submartingale.
As a consequence we have
\begin{equation}\label{submartineq}
\P\Big(\sup_{0\le k\le m}S_k\ge xm\Big)=\P\Big(\sup_{0\le k\le m}e^{uS_k}\ge e^{uxm}\Big)\le\frac{\E\left(e^{uS_m}\right)}{e^{uxm}},
\end{equation}
where we used the submartingale inequality in the last step.

The rest of the proof of \eqref{ldpsup} is a standard large deviation argument.
The expectation on the right-hand side of \eqref{submartineq} is the moment generating function of the random walk.
The moment generating function of the two types of steps given in \eqref{defrwsteps} are
\begin{equation}
\E(e^{uX_1})=\frac{\sqrt2-1+e^u}{\sqrt2},\qquad\E(e^{uX_2})=\frac{2-\sqrt2}{1-(\sqrt2-1)e^{-u}}
\end{equation}
for any $u>\log(\sqrt2-1)$.
Since \eqref{submartineq} holds for any $u>0$, to optimize in $u$, one computes the Legendre transform of the logarithmic moment generating function of one double step with the restriction $u>0$.
Hence \eqref{ldpsup} holds with
\begin{equation}
I(x)=\sup_{u>0}\left(ux-\log\left(\E\left(e^{uX_1}\right)\E\left(e^{uX_2}\right)\right)\right).
\end{equation}
The optimal $u$ without the positivity restriction is given by $u=\log((\sqrt2-1)(1+\sqrt{1+(1-x)^2})/(1-x))$ which turns out to be positive for all $x>0$.
Then \eqref{defIx} for $I(x)$ follows by computation.

To show the lower bound \eqref{Ixsmallbound}, one observes that Taylor expansion of $I(x)$ yields a quadratic lower bound in a small neighbourhood of $0$.
Since $I(x)$ is convex as a large deviation rate function, the lower bound can be extended to $(0,1]$ by choosing the coefficient of the quadratic term small enough.
Since $I(x)=\infty$ for $x>1$, \eqref{Ixsmallbound} follows for any $x\ge0$.
\end{proof}

\section{Reformulation with hitting times}\label{s:hitting}

We prove Theorem~\ref{thm:ReformulationKernel} in this section and we give a few examples where it can directly be used.
Let
\begin{equation}\label{defphi}
\phi_t(x,y)=\frac1{\sqrt{4\pi t}}e^{-\frac{(x-y)^2}{4t}}.
\end{equation}
be the Brownian transition kernel of diffusion coefficient $2$ which we use in the proof below.

\begin{proof}[Proof of Theorem~\ref{thm:ReformulationKernel}]
Let us define
\begin{equation}\label{eqA1}
R(u,v)= (e^{-(M-L) H})(u,v)\P_{b(L)=u-L^2,b(M)=v-M^2}\left(\exists t\in [L,M] \textrm{ s.t. }b(t)> g(t)-t^2\right).
\end{equation}
With the notations $A(x,y)=\Ai(x+y)$ and $P_0(x)=\id_{x\geq 0}$ we have $K_{\rm Ai}= A P_0 A$.
Inserting these definitions and using the identity $\det(\id-A B)=\det(\id-B A)$, we get that the left-hand side of \eqref{CQR} is equal to
\begin{equation}
\det(\id-K_{\rm Ai} + \Theta^g_{L,M} e^{(M-L) H} K_{\rm Ai})_{L^2(\R)} = \det(\id-A R e^{(M-L)H} A)_{L^2(\R_+)}
\end{equation}
where the $P_0$ is absorbed in the definition of the space (from $L^2(\R)$ to $L^2(\R_+)$).

Then we have obtained \eqref{ReformulationKernel} with the kernel on the right-hand side given by the conjugated kernel
\begin{equation}\label{conjugatedkernel}
K(x,y)=\frac{e^{Lx}}{e^{Ly}}(A R e^{(M-L)H} A)(x,y).
\end{equation}
We are left with identifying the kernel to be equal to \eqref{defKreformulated}.
The kernel in \eqref{conjugatedkernel} is given explicitly by
\begin{equation}\label{eqA6}
K(x,y)=\frac{e^{Lx}}{e^{Ly}}\int_\R\d u \int_\R\d v \Ai(x+u) R(u,v) \Ai(y+v) e^{-(M-L) y}.
\end{equation}
It remains to find a good expression for $K$ from which the $L\to -\infty$ and $M\to\infty$ limit are easily taken, see also Remark~\ref{remark3.2}. We can write
\begin{equation}\label{eqA8}
R(u,v)=\frac{e^{-(M-L)H}(u,v) \P_{b(L)=u-L^2}(\exists \tau\in [L,M] \textrm{ s.t. }b(\tau)> g(\tau)-\tau^2, b(M)=v-M^2)}{\phi_{M-L}(u-L^2,v-M^2)}
\end{equation}
with the notation \eqref{defphi}. The last term is the probability \emph{density} that the Brownian bridge starting from $b(L)=u-L^2$ reaches $b(M)=v-M^2$ and crosses curve $g(t)-t^2$ somewhere in between.
An explicit computation gives (c.f.\ $g=-\infty$ in Theorem~\ref{thm:CQR})
\begin{equation}\label{eqA9}
\frac{e^{-(M-L)H}(u,v)}{\phi_{M-L}(u-L^2,v-M^2)} = e^{(M^3-L^3)/3+Lu-Mv}.
\end{equation}
Next we decompose the probability of crossing depending on whether the inequality is not satisfied to the left and/or right of the origin.
Let us define
\begin{equation}\begin{aligned}
m_-^\alpha(u-L^2,\xi)&=\P_{b(L)=u-L^2}(\exists \tau\in [L,\alpha] \textrm{ s.t. }b(\tau)> g(\tau)-\tau^2, b(\alpha)=\xi),\\
m_+^\alpha(\xi,v-M^2)&=\P_{b(\alpha)=\xi}(\exists \tau\in [\alpha,M] \textrm{ s.t. }b(\tau)> g(\tau)-\tau^2, b(M)=v-M^2).
\end{aligned}\end{equation}
By inclusion--exclusion, we have
\begin{equation}\label{eqA11}\begin{aligned}
& \P_{b(-L)=u-L^2}(\exists \tau\in [L,M] \textrm{ s.t. }b(\tau)> g(\tau)-\tau^2, b(M)=v-M^2)\\
&\quad = \int_\R\d\xi m_-^\alpha(u-L^2,\xi) \phi_{M-\alpha}(\xi,v-M^2)+ \int_\R\d\xi \phi_{\alpha-L}(u-L^2,\xi)m_+^\alpha(\xi,v-M^2)\\
&\qquad -\int_\R\d\xi m_-^\alpha(u-L^2,\xi) m_+^\alpha(\xi,v-M^2).
\end{aligned}\end{equation}
Plugging in \eqref{eqA8}, \eqref{eqA9}, and \eqref{eqA11} into \eqref{eqA6} and by doing the change of variables $u\to u+L^2$ and $v\to v+M^2$ we obtain
\begin{multline}\label{eqA12}
K(x,y)=\int_\R\d u \int_\R\d v \Ai^{(L)}(x+u) \Ai^{(-M)}(y+v) \\
\times\left(\int\d\xi m_-^\alpha(u,\xi) \phi_{M-\alpha}(\xi,v)+ \int\d\xi \phi_{\alpha-L}(u,\xi)m_+^\alpha(\xi,v) -\int\d\xi m_-^\alpha(u,\xi) m_+^\alpha(\xi,v)\right).
\end{multline}

We can express now $m_-^\alpha(u,\xi)$ and $m_+^\alpha(\xi,v)$ by integrating over the hitting times and their positions as follows:
\begin{equation}\begin{aligned}
m_-^\alpha(u,\xi)&=\int_{L}^\alpha \int_\R \P(T_-^{\xi,\alpha}\in\d t, X_-^{\xi,\alpha}\in\d\zeta) \phi_{t-L}(u,\zeta),\\
m_+^\alpha(\xi,v)&=\int_{\alpha}^M \int_\R \P(T_+^{\xi,\alpha}\in\d t, X_+^{\xi,\alpha}\in\d\zeta) \phi_{M-t}(\zeta,v).
\end{aligned}\end{equation}
Noting that when $\P(T_+^{\xi,\alpha}\in\d t, X_+^{\xi,\alpha}\in\d\zeta)$ is a Dirac distribution at $t=\alpha,\zeta=\xi$, one recovers $m_+^\alpha(\xi,v)=\phi_{M-\alpha}(\xi,v)$.
Thus we can compute first the last term in \eqref{eqA12}, while the first two cases are recovered as special cases. Using Lemma~\ref{lem:app} below, we get the identities
\begin{equation}\label{eqA16}
\int_\R\d u \Ai^{(L)}(x+u)\phi_{t-L}(u,\zeta) = \Ai^{(t)}(x+\zeta),\quad \int_\R\d v \Ai^{(-M)}(y+v) \phi_{M-t}(\zeta,v) = \Ai^{(-t)}(y+\zeta).
\end{equation}
Integrating over $u$ and $v$ in \eqref{eqA12} using \eqref{eqA16} we get the claimed formula.
\end{proof}

\begin{lemma}\label{lem:app}
With the notation \eqref{defphi},
\begin{equation}\label{AiPhi}
\int_\R\d u \Ai^{(s)}(x+u)\phi_{t-s}(u,y)=\Ai^{(t)}(x+y).
\end{equation}
\end{lemma}
\begin{proof}
The identity is obtained by first expressing the Airy function as complex integral
\begin{equation}
\Ai^{(t)}(x)=\Ai(t^2+x)e^{2t^3/3+tx}=\frac{1}{2\pi i}\int_\langle \d w e^{w^3/3+tw^2-xw}
\end{equation}
and by computing the Gaussian integration in $u$.
\end{proof}

\subsubsection*{Examples}

\paragraph{The $L=\alpha=0$ case.}
Consider the special case $L=\alpha=0$. Then $M_-^\alpha(x,\xi)=\Ai(\xi+x)\id_{\xi>g(0)}$ and we are left with
\begin{equation}
K(x,y)=\int_{-\infty}^{g(0)}\d\xi \Ai(\xi+x) M_+^\alpha(\xi,y)+\int_{g(0)}^\infty\d\xi \Ai(\xi+x)\Ai(\xi+y).
\end{equation}

\paragraph{One-point barrier.}
Let $L=\alpha=0$ and $t_0>0$.
Consider $g(t_0)=a\in\R$ and $g(t)=\infty$ for $t\neq t_0$. Then
\begin{equation}
K(x,y)=\int_\R\d\xi \Ai(\xi+x) \int_{a-t_0^2}^\infty\d\zeta \phi_{t_0}(\xi,\zeta) \Ai^{(-t_0)}(\zeta+y).
\end{equation}
Using Lemma~\ref{lem:app} we get
\begin{equation}
K(x,y)=\int_{a-t_0^2}^\infty\d\zeta \Ai^{(t_0)}(x+\zeta)\Ai^{(-t_0)}(y+\zeta)=e^{t_0(x-y)} \int_{a}^\infty\d\zeta \Ai(x+\zeta)\Ai(y+\zeta).
\end{equation}
This gives
\begin{equation}
\det(\id-K)_{L^2(\R_+)}=\det(\id-K_{\rm Ai})_{L^2(a,\infty)}=F_{\rm GUE}(a).
\end{equation}

\paragraph{Flat cut-off.}
The flat cut-off in the original system corresponds to the choice $g(t)=R+t^2$ for some fixed cut-off value $R$.
In this case it is possible to take $L\to-\infty$ and $M\to\infty$ without problems, see also Remark~\ref{remark3.2}.
We get
\begin{equation}
M_+^\alpha(\xi,y)=\Ai^{(-\alpha)}(\xi+y)\id_{\xi\geq R} + \int_\alpha^\infty \P(T_+^\xi\in\d t) \Ai^{(-t)}(R+y)\id_{\xi<R}.
\end{equation}
The reflection principle gives, for $\xi<R$ and $t>\alpha$,
\begin{equation}
\P(T_+^\xi\in\d t) = (R-\xi) \frac{e^{-(R-\xi)^2/(4(t-\alpha))}}{\sqrt{4\pi (t-\alpha)^3}}\d t.
\end{equation}
Using the integral representation $\Ai^{(-t)}(x)=\Ai(t^2+x) e^{-2 t^3/3-t x}= \int_{\langle} \frac{\d w}{2\pi i} e^{w^3/3-t w^2-x w}$ we can first integrate explicitly over $t$ with the result
\begin{equation}\label{Mpmflat}
\begin{aligned}
M_+^\alpha(\xi,y)&=\Ai^{(-\alpha)}(\xi+y)\id_{\xi\geq R} + \Ai^{(-\alpha)}(2R+y-\xi)\id_{\xi<R},\\
M_-^\alpha(x,\xi)&=\Ai^{(\alpha)}(\xi+x)\id_{\xi\geq R} + \Ai^{(\alpha)}(2R+x-\xi)\id_{\xi<R}.
\end{aligned}
\end{equation}
Plugging in the formula of the kernel after some simple cancellations one obtains
\begin{equation}
K(x,y)=\int_\R\d\xi \Ai^{(\alpha)}(R+x-\xi)\Ai^{(-\alpha)}(R+y+\xi) = 2^{-1/3} \Ai\left(2^{-1/3}(2R+x+y)\right)
\end{equation}
where the last identity is slightly tricky (use the integral representation with vertical contours, once with one ordering so that the integral over $\xi\in\R_+$ is convergent,
once with the other order so that the integral over $\xi\in\R_-$ is convergent. Their sum can be computed then with the residue theorem leading to the identity).
Thus we have the well-known identity~\cite{Jo01,CQR11,FS05b}
\begin{equation}
\P({\cal A}_2(t)-t^2\leq R\textrm{ for all }t\in\R)=\det(\id-K)_{L^2(\R_+)} = F_{\rm GOE}(2^{2/3} R).
\end{equation}

\section{Left endpoint approach}\label{s:leftepa}

As already noticed in~\cite{CQR11}, but clearly pointed out in~\cite{QR19},
the representation in \eqref{CQR} is not adequate for taking $L\to-\infty$ and/or $M\to\infty$ as some of the terms taken individually do not have a limit.
In~\cite{QR19} they introduced a decomposition with respect to the position taken by the Brownian bridge at an intermediate time (e.g., at time $0$) and the kernel was rewritten in terms of hitting times.
This approach allowed to take the desired limits relatively directly.

In this paper we first show the convergence of probabilities about the top line of a discrete line ensemble to that of the Airy$_2$ process on the interval $[L,M]$ for fixed $L$ and $M$.
Then we prove that when $L\to-\infty$ and $M\to\infty$ we recover the problem for the full-line case.
One possibility would be to apply some probabilistic bounds in the spirit of~\cite{CLW16} or~\cite{CFS16}, by using the correspondence with the discrete time TASEP with parallel update.
However, in this paper we extend the convergence to the Airy$_2$ process to infinite intervals using the path integral formulation, see Section~\ref{s:extension}.
Surprisingly, it turns out that using the strategy of~\cite{QR19} with two hitting times generates issues in the asymptotic analysis of the backward part of the random walk,
therefore it is more suitable for our purposes to introduce hitting times of the random walks starting from the left endpoint of the interval only.

The second statement \eqref{belowcurverewriteAiry} of Theorem~\ref{thm:belowcurverewrite} below is a direct consequence of Theorem~\ref{thm:ReformulationKernel}.
We mention that a discrete analogue of Theorem~\ref{thm:ReformulationKernel} could be derived
for the probability on the left-hand side of \eqref{belowcurverewrite} in terms of hitting times for a random walk in two directions.
It is however used only in the case when the starting point of the random walks is the left endpoint of the interval when the formulas simplify.
For this reason we directly prove \eqref{belowcurverewrite} using a random walk which starts at the left endpoint.

Define the hitting time and position when the random walk is above the curve $g_n$ by
\begin{equation}\label{defTXhat}
\wh T^{u,m}_+=\min\{l\ge m:S_l>g_n(b_n^{-1}(2l))\mbox{ with }S_m=u\},\qquad\wh X^{u,m}_+=S_{\wh T^{u,m}_+}.
\end{equation}
Then let the discrete kernel be
\begin{equation}\label{defKnLM}
K_n^{L,M}(i,j)=\sum_{u\in\Z}\sum_{l=b_n(L)/2}^{b_n(M)/2}\sum_{v\in\Z}\P\left(\wh T^{u,\frac{b_n(L)}2}_+=l,\wh X^{u,m}_+=v\right)k_n^{u,l,v}(i,j)
\end{equation}
where
\begin{equation}\label{defk}
k_n^{u,l,v}(i,j)=(\sqrt2+1)^{i-j-b_n(L)+2l+u-v}\,q_{b_n(L)/2}^{(n)}(u+i)\,p_l^{(n)}\left(v+j\right).
\end{equation}
Let the continuous kernel be
\begin{equation}\label{defKLM}
K^{L,M}(x,y)=\int_\R\d\xi\Ai^{(L)}(x+\xi)\int_L^M\int_\R\P(T^{\xi,L}_+\in\d t,X^{\xi,L}_+\in\d\zeta)\Ai^{(-t)}(\zeta+y).
\end{equation}

\begin{theorem}\label{thm:belowcurverewrite}
Fix $L<M$.
Then
\begin{equation}\label{belowcurverewrite}
\P\left(Y_n(b_n(\tau))\le g_n(\tau)\mbox{ for }\tau\in[L,M]\right)=\det\left(\id-K_n^{L,M}\right)_{\ell^2(\Z_+)}
\end{equation}
with the kernel $K_n^{L,M}$ defined in \eqref{defKnLM}.
Furthermore,
\begin{equation}\label{belowcurverewriteAiry}
\P\left(\mathcal A_2(\tau)\le g(\tau)\mbox{ for }\tau\in[L,M]\right)=\det\left(\id-K^{L,M}\right)_{L^2(\R_+)}
\end{equation}
with the kernel $K^{L,M}$ given by \eqref{defKLM}.
\end{theorem}

For the proof of Theorem~\ref{thm:belowcurverewrite} we will use the following properties as well.
\begin{lemma}\label{lemma:Tpq}
It holds
\begin{align}
(Tp_r^{(n)})(x)&=p_{r-1}^{(n)}(x)\label{Tp},\\
(q_s^{(n)}T)(y)&=q_{s+1}^{(n)}(y)\label{qT}.
\end{align}
\end{lemma}

\begin{proof}[Proof of Lemma~\ref{lemma:Tpq}]
By \eqref{Trepr} and definition \eqref{defp},
\begin{equation}
(Tp_r^{(n)})(x)=(\sqrt2+1)^2\,\E\left(\frac{-1}{2\pi i}\oint_{\Gamma_1}\frac{\d z}z\frac1{((\sqrt2+1)z)^{x+X}(1-z)^{n-r}(1+1/z)^r}\right)
\end{equation}
holds.
The order of the integration and the expectation can be exchanged,
because the integrand is absolutely integrable with respect to the product measure which can be seen by noting that $X\le1$.
By taking the expectation of the integrand above, one encounters the generating function of $X$.
For the random variable $X=X_1+X_2$ given by \eqref{defX}, the generating function is given by
\begin{equation}
\E\left(s^X\right)=(\sqrt2-1)\frac{s(s+\sqrt2-1)}{s-(\sqrt2-1)}
\end{equation}
for any $|s|>\sqrt2-1$.
After substituting $s=1/((\sqrt2+1)z)$, we get
\begin{equation}
\E\left(\frac1{((\sqrt2+1)z)^X}\right)=(\sqrt2+1)^{-2}\frac{1+1/z}{1-z}
\end{equation}
which proves \eqref{Tp}.
The proof of \eqref{qT} is similar after observing that \eqref{Tkernel} can be used to write
\begin{equation}\begin{aligned}
(q_s^{(n)}T)(y)&=\sum_{x\in\Z}q_s^{(n)}(x)(\sqrt2+1)^{2-y+x}\P(X=y-x)\\
&=(\sqrt2+1)^2\E\left((\sqrt2+1)^{-X}q_s^{(n)}(y-X)\right).
\end{aligned}\end{equation}
\end{proof}

\begin{proof}[Proof of Theorem~\ref{thm:belowcurverewrite}]
First we rephrase the condition on the left-hand side of \eqref{belowcurverewrite} in a way that the top line $Y_n$ remains below the given curve $g_n$ after each double step:
\begin{equation}\label{belowcurverewrite2}\begin{aligned}
&\P\left(Y_n(b_n(\tau))\le g_n(\tau)\mbox{ for }\tau\in[L,M]\right)\\
&\quad=\P\bigg(\bigcap_{l=b_n(L)/2}^{b_n(M)/2}\left\{Y_n(l)\le g_n\left(b_n^{-1}(2l)\right)\right\}\bigg)\\
&\quad=\det\Big(\id-\wt K_n(b_n(L),\cdot;b_n(L),\cdot)\\
&\quad \hspace{4em}+\ol P_{x_{b_n(L)/2}}T\ol P_{x_{b_n(L)/2+1}}\cdots T\ol P_{x_{b_n(M)/2}}\wt K_n(b_n(M),\cdot;b_n(L),\cdot)\Big)_{\ell^2(\Z)}\\
&\quad=\det\left(\id-R_nT^{-\frac{b_n(M)-b_n(L)}2}\wt K_n(b_n(L),\cdot;b_n(L),\cdot)\right)_{\ell^2(\Z)}
\end{aligned}\end{equation}
where $x_l=g_n(b_n^{-1}(2l))$ for $l=b_n(L)/2,b_n(L)/2+1,\dots,b_n(M)/2$ and
\begin{multline}\label{defRn}
R_n(i,j)=T^{\frac{b_n(M)-b_n(L)}2}(i,j)\\
\times\P_{S_{\frac{b_n(L)}2}=i,S_{\frac{b_n(M)}2}=j}\left(\exists l\in\left\{\tfrac12 b_n(L),\tfrac12 b_n(L)+1,\dots,\tfrac12 b_n(M)\right\}:S_l>g_n\left(b_n^{-1}(2l)\right)\right)
\end{multline}
is the kernel on the right-hand side.
In the last two equalities in \eqref{belowcurverewrite2}, we used Proposition~\ref{prop:Yrw}.

By \eqref{kernelrewrite}, we can write $\wt K_n(b_n(L),x;b_n(L),y)=(\pi_n^LP_0\rho_n^L)(x,y)$ where
\begin{equation}\label{defpirho}
\pi_n^L(x,j)=p_{\frac{b_n(L)}2}^{(n)}(x+j),\qquad\rho_n^L(j,y)=q_{\frac{b_n(L)}2}^{(n)}(y+j).
\end{equation}
Using this and the determinant identity $\det(\id+AB)=\det(\id+BA)$, we obtain
\begin{equation}\label{belowcurverewrite3}
\det\left(\id-R_nT^{-\frac{b_n(M)-b_n(L)}2}\wt K_n(b_n(L),\cdot;b_n(L),\cdot)\right)_{\ell^2(\Z)}
\hspace{-1em}=\det\left(\id-\rho_n^LR_nT^{-\frac{b_n(M)-b_n(L)}2}\pi_n^L\right)_{\ell^2(\Z_+)}\hspace{-1em}.
\end{equation}

Next we rewrite \eqref{defRn} using \eqref{Tkernel} and by decomposing the crossing event according to the first hitting time $\wh T^{i,\frac{b_n(L)}2}_+$:
\begin{equation}\begin{aligned}
&R_n(i,j)\\
&\quad=(\sqrt2+1)^{b_n(M)-b_n(L)-j+i}\\
&\qquad\times\P_{S_{\frac{b_n(L)}2}=i}\left(\left\{\exists l\in\left\{\tfrac12 b_n(L),\dots,\tfrac12 b_n(M)\right\}:S_l>g_n\left(b_n^{-1}(2l)\right)\right\}\cap\{S_{\frac{b_n(M)}2}=j\}\right)\\
&\quad=(\sqrt2+1)^{b_n(M)-b_n(L)-j+i}\sum_{l=b_n(L)/2}^{b_n(M)/2}\P\Big(\wh T^{i,\frac{b_n(L)}2}_+=l,\wh X^{i,\frac{b_n(L)}2}_+=v\Big)\P_{S_l=v}\Big(S_{\frac{b_n(M)}2}=j\Big)\\
&\quad=(\sqrt2+1)^{b_n(M)-b_n(L)-j+i}\\
&\qquad\times\sum_{l=b_n(L)/2}^{b_n(M)/2}\P\Big(\wh T^{i,\frac{b_n(L)}2}_+=l,\wh X^{i,\frac{b_n(L)}2}_+=v\Big)(\sqrt2+1)^{-b_n(M)+2l-v+j}T^{\frac{b_n(M)}2-l}\left(v,j\right).
\end{aligned}\end{equation}
By applying \eqref{Tp} to $\pi_n^L$ given in \eqref{defpirho}, one gets that the conjugated kernel
\begin{equation}\label{rhoRTpi}
(\sqrt2+1)^{i-j}\rho_n^LR_nT^{-\frac{b_n(M)-b_n(L)}2}\pi_n^L(i,j)=K_n^{L,M}(i,j).
\end{equation}
Putting together \eqref{belowcurverewrite2}, \eqref{belowcurverewrite3} and \eqref{rhoRTpi} we obtain \eqref{belowcurverewrite}.

For proving \eqref{belowcurverewriteAiry}, we specialize Theorem~\ref{thm:ReformulationKernel} for the case when $\alpha=L$.
In that case by \eqref{defM},
\begin{equation}
M^L_-(x,\xi)=\Ai^{(L)}(x+\xi)\id_{\xi>g(L)-L^2}
\end{equation}
and $M^L_+(\xi,y)=\Ai^{(-L)}(\xi+y)$ if $\xi>g(L)-L^2$.
Therefore, the first and third integral on the right-hand side of \eqref{defKreformulated} cancel out and the second term is equal to the kernel $K^{L,M}$ defined in \eqref{defKLM} which completes the proof.
\end{proof}

\section{Asymptotic statements}\label{s:asympstate}

In this section, we consider the asymptotics when the size of the Aztec diamond $n\to\infty$ and prove the finite interval version of Theorem~\ref{thm:YbelowacurveR}.

\begin{theorem}\label{thm:Ybelowacurve}
For any $L<M$ fixed and any function $g$ with the assumptions of Theorem~\ref{thm:Xbelowh},
\begin{equation}\label{Ybelowacurve}
\P\left(Y_n(b_n(\tau))\le g_n(\tau)\mbox{ for }\tau\in[L,M]\right)\to\P\left(\mathcal A_2(\tau)\le g(\tau)\mbox{ for }\tau\in[L,M]\right)
\end{equation}
as $n\to\infty$ where $g_n$ is given by \eqref{defbtaugn} and $\mathcal A_2$ is the Airy$_2$ process.
\end{theorem}

For the proof, we use the representations given in Theorem~\ref{thm:belowcurverewrite} by the left endpoint approach.
In particular, we show that the Fredholm determinant of $K_n^{L,M}$ on the right-hand side of \eqref{belowcurverewrite} in Theorem~\ref{thm:belowcurverewrite}
converges to that of $K^{L,M}$ \eqref{belowcurverewriteAiry}.
The convergence of Fredholm determinants is based on the following series of propositions which we prove in Section~\ref{s:asympproofs}. To simplify the notations, define
\begin{equation}
\begin{aligned}
P_l^{(n)}(v+j)&=2^{-5/6}n^{1/3}2^{n/2}(\sqrt2+1)^{-j-2^{-5/6}\zeta n^{1/3}+2^{-1/6}tn^{2/3}}p_l^{(n)}\left(v+j\right),\\
Q_{L}^{(n)}(u+i)&=2^{-5/6}n^{1/3}2^{-n/2}(\sqrt2+1)^{i+2^{-5/6}\xi n^{1/3}-2^{-1/6}Ln^{2/3}}q_{\frac{b_n(L)}2}^{(n)}(u+i).
\end{aligned}
\end{equation}

\begin{proposition}\label{prop:pqconv}
Under the scaling
\begin{equation}\label{scaling}\begin{gathered}
l=n\left(\frac12+\frac1{2\sqrt2}\right)+2^{-7/6}tn^{2/3},\qquad i=2^{-5/6}xn^{1/3},\qquad j=2^{-5/6}yn^{1/3},\\
u=\frac n{\sqrt2}+2^{-5/6}\xi n^{1/3},\qquad v=\frac n{\sqrt2}+2^{-5/6}\zeta n^{1/3},
\end{gathered}\end{equation}
the two convergence statements
\begin{align}
P_l^{(n)}(v+j)&\to\Ai^{(-t)}(\zeta+y),\label{pconv}\\
Q_{L}^{(n)}(u+i)&\to\Ai^{(L)}(x+\xi)\label{qconv}
\end{align}
hold uniformly on compact intervals in $\zeta+y$ and in $x+\xi$ respectively as $n\to\infty$.
\end{proposition}

\begin{proposition}\label{prop:pbound}
Let $L<M$ be fixed and consider the scaling \eqref{scaling} of the variables.
Then there are $c>0$ and $C\in\R$ such that for all $n$ large enough the bound
\begin{equation}\label{pbound}
\big|P_l^{(n)}(v+j)\big|\le Ce^{-c(y+\zeta)}
\end{equation}
holds for all $y\ge0$ and $\zeta$ bounded from below uniformly in $l\in[b_n(L)/2,b_n(M)/2]$.
\end{proposition}

\begin{proposition}\label{prop:qbound}
Let $L<M$ be fixed and consider the scaling \eqref{scaling}.
There are $c>0$ and $C\in\R$ so that for all $n$ large enough
\begin{equation}\label{qbound}
\big|Q_{L}^{(n)}(u+i)\big|\le Ce^{-c(x+\xi)}
\end{equation}
with $x\ge0$ and $\xi\in\R$.
\end{proposition}

\begin{proposition}\label{prop:improvedbound}
Under the scaling \eqref{scaling} of the variables, for any $c>0$ there is a constant $C\in\R$ such that for all $n$ large enough
\begin{equation}\label{improvedbound}
\big|P_l^{(n)}(v+j)\big|\le Ce^{-c(t+y+\zeta)}
\end{equation}
holds for all $t\ge0$, $y\ge0$ and $\zeta$ bounded from below.
\end{proposition}

\begin{proposition}\label{prop:Hconv}
Suppose that $g:\R\to\R$ is a function as in Theorem~\ref{thm:Xbelowh}.
Under the scaling \eqref{scaling}, the rescaled hitting time and hitting position of the random walk given in \eqref{defTXhat} converge jointly weakly
\begin{equation}\label{Hconv}
\left(2^{7/6}n^{-2/3}\wh T^{u,\frac{b_n(L)}2}_+,2^{5/6}n^{-1/3}\wh X^{u,\frac{b_n(L)}2}_+\right)\Rightarrow\left(T^{\xi,L}_+,X^{\xi,L}_+\right)
\end{equation}
holds as $n\to\infty$ where the limit is given in \eqref{defTX+}.
\end{proposition}

\begin{lemma}\label{lemma:intkerneldecay}
Let $L\in\R$ be fixed.
There are $c>0$ and $C\in\R$ such that for any $T>L$
\begin{equation}\label{kdecay}
\sum_{u\in\Z}\sum_{l=\frac{b_n(T)}2}^{\frac{b_n(T+1)}2}\sum_{v\in\Z}\left|2^{-5/6}n^{1/3}\P\left(\wh T^{u,\frac{b_n(L)}2}_+=l,\wh X^{u,m}_+=v\right)k_n^{u,l,v}(i,j)\right|\le Ce^{-c(x+y+T)}
\end{equation}
holds for any $x,y\ge0$ uniformly for all $n$ large enough.
Further, for any $T>L$
\begin{equation}\label{kappadecay}
\int_\R\d\xi\int_T^{T+1}\int_\R\P(T^{\xi,L}_+\in\d t,X^{\xi,L}_+\in\d\zeta)\left|\Ai^{(L)}(x+\xi)\Ai^{(-t)}(\zeta+y)\right|\le Ce^{-c(x+y+T)}
\end{equation}
holds for any $x,y\ge0$.
\end{lemma}

\begin{proof}[Proof of Theorem~\ref{thm:Ybelowacurve}]
First we use Theorem~\ref{thm:belowcurverewrite} to represent both sides of \eqref{Ybelowacurve} as Fredholm determinants.
We introduce the rescaled discrete kernel
\begin{equation}\label{defKresc}
K_n^{L,M,\text{resc}}(x,y)=2^{-5/6}n^{1/3}K_n^{L,M}\left(2^{-5/6}n^{1/3}x,2^{-5/6}n^{1/3}y\right).
\end{equation}
Next we show that for any $x,y\ge0$
\begin{equation}\label{kernelconv}
K_n^{L,M,\text{resc}}(x,y)\to K^{L,M}(x,y)
\end{equation}
as $n\to\infty$ and that there are $c>0$ and $C\in\R$ such that
\begin{equation}\label{kerneldecay}
\left|K_n^{L,M,\text{resc}}(x,y)\right|\le Ce^{-c(x+y)}
\end{equation}
holds for any $x,y\ge0$.
Then the convergence of Fredholm determinant follows from \eqref{kernelconv} and \eqref{kerneldecay} by dominated convergence.

Proposition~\ref{prop:pqconv} yields that under the scaling \eqref{scaling}
\begin{equation}\label{kconv}
2^{-5/3}n^{2/3}k_n^{u,l,v}(i,j)\to\Ai^{(L)}(x+\xi)\Ai^{(-t)}(\zeta+y)
\end{equation}
holds as $n\to\infty$.
By the weak convergence in Proposition~\ref{prop:Hconv} and by \eqref{kdecay} in Lemma~\ref{lemma:intkerneldecay}
dominated convergence implies \eqref{kernelconv} and \eqref{kerneldecay} for the $L=T$ and $M=T+1$ case with the right-hand side of \eqref{kerneldecay} replaced by $Ce^{-c(x+y+T)}$.
The general $L<M$ case follows immediately which proves the theorem.
\end{proof}

\section{Extension of conditioning}\label{s:extension}

In this section, we prove that the conditioning for the top path of the Aztec diamond ensemble and that of the Airy process can be extended to the whole $\R$,
that is, we prove Theorem~\ref{thm:YbelowacurveR} from its finite interval counterpart Theorem~\ref{thm:Ybelowacurve}.

Let us define
\begin{align}
a_n^{L,M}&=\P(Y_n(b_n(\tau))\le g_n(\tau)\mbox{ for }\tau\in[L,M]),\label{defanLM}\\
a^{L,M}&=\P(\mathcal A_2(\tau)\le g(\tau)\mbox{ for }\tau\in[L,M])\label{defaLM}
\end{align}
where the cases $L=-\infty$ and $M=\infty$ are also allowed.
For $L,M\in\R$,
\begin{equation}\label{aFredholm}
a_n^{L,M}=\det\left(\id-K_n^{L,M}\right)_{\ell^2(\Z_+)},\qquad a^{L,M}=\det\left(\id-K^{L,M}\right)_{L^2(\R_+)}
\end{equation}
holds by Theorem~\ref{thm:belowcurverewrite} where the kernels are given by \eqref{defKnLM} and \eqref{defKLM}.
In the rest of this section, we prove that \eqref{aFredholm} can be extended for $M=\infty$ with fixed $L\in\R$ and the kernels in the Fredholm determinant formulas make sense for these values.
We mention however that the kernel inside the Fredholm determinant formulas in \eqref{aFredholm} are not well-defined for $L=-\infty$.

\begin{proposition}\label{prop:halfinfiniteFredholm}
Let $L\in\R$ be fixed.
Then
\begin{equation}\label{KnLMconv}
\det(\id-K_n^{L,M})_{\ell^2(\Z_+)}\to\det(\id-K_n^{L,\infty})_{\ell^2(\Z_+)}
\end{equation}
uniformly in $n$ as $M\to\infty$ where the kernel $K_n^{L,\infty}$ is obtained from $K_n^{L,M}$ defined in \eqref{defKnLM} by the formal substitution $M=\infty$.
As a consequence, for any positive integer $n$,
\begin{equation}\label{halfinfiniteFredholm}
\P(Y_n(b_n(\tau))\le g_n(\tau)\mbox{ for }\tau\in[L,\infty))=\det\left(\id-K_n^{L,\infty}\right)_{\ell^2(\Z_+)}.
\end{equation}
\end{proposition}

\begin{proof}[Proof of Proposition~\ref{prop:halfinfiniteFredholm}]
First note that by change of variables the scaling identity
\begin{equation}\label{kernelscaling}
\det\left(\id-K_n^{L,M}\right)_{\ell^2(\Z_+)}=\det\left(\id-K_n^{L,M,\text{resc}}\right)_{L^2(\R_+)}
\end{equation}
holds where the rescaled kernel is given by \eqref{defKresc}.
The identity \eqref{kernelscaling}--\eqref{defKresc} also holds when $M$ is replaced by $\infty$.
To show the uniform convergence in \eqref{KnLMconv}, we use the general bound on the difference of Fredholm determinants for our case
\begin{multline}\label{Fredholmdiff}
\left|\det(\id-K_n^{L,M,\text{resc}})_{L^2(\R_+)}-\det(\id-K_n^{L,\infty,\text{resc}})_{L^2(\R_+)}\right|\\
\le\left\|K_n^{L,M,\text{resc}}-K_n^{L,\infty,\text{resc}}\right\|_1\exp\left(\left\|K_n^{L,M,\text{resc}}\right\|_1+\left\|K_n^{L,\infty,\text{resc}}\right\|_1+1\right).
\end{multline}
As a direct consequence of \eqref{defKnLM} and \eqref{kdecay} in Lemma~\ref{lemma:intkerneldecay}, we get that
\begin{equation}
\left|K_n^{L,M,\text{resc}}(x,y)-K_n^{L,\infty,\text{resc}}(x,y)\right|\le Ce^{-c(x+y+M)}.
\end{equation}
For bounding the $1$-norm of the kernels on the right-hand side of \eqref{Fredholmdiff}, let us define the kernel $B(x,y)=\delta_{x,y}e^{-cx/2}$.
Then using the fact that the $1$ norm of a product of kernels can be upper bounded by the product of the $2$-norms of the factors, we can write
\begin{equation}\begin{aligned}
\left\|K_n^{L,M,\text{resc}}-K_n^{L,\infty,\text{resc}}\right\|^2_1&\le\left\|B\right\|^2_2\cdot\left\|B^{-1}\left(K_n^{L,M,\text{resc}}-K_n^{L,\infty,\text{resc}}\right)\right\|_2^2\\
&\le\left(\int_0^\infty\d x\,e^{-cx}\right)Ce^{-2cM}\int_0^\infty\d x\int_0^\infty\d y\,e^{-2(cx/2+cy)}\\
&\le C'e^{-2cM},
\end{aligned}\end{equation}
which holds uniformly in $n$ proving uniform convergence in \eqref{KnLMconv}.

Next we let $M\to\infty$ in the first equation of \eqref{aFredholm} for fixed $L$.
For fixed $L$ and $n$, the events $\{Y_n(b_n(\tau))\le g_n(\tau)\mbox{ for }\tau\in[L,M]\}$ form a decreasing family in $M$,
hence by the continuity of measure, their probabilities converge, i.e.\ $\lim_{M\to\infty}a_n^{L,M}=a_n^{L,\infty}$ for any $L$ and $n$.
Therefore by \eqref{KnLMconv}, we see that $a_n^{L,\infty}=\det\left(\id-K_n^{L,\infty}\right)_{\ell^2(\Z_+)}$ which proves \eqref{halfinfiniteFredholm}.
\end{proof}

\begin{proposition}\label{prop:halfinfiniteconv}
For any fixed $L\in\R$, it holds
\begin{equation}\label{halfinfiniteconv}
\lim_{n\to\infty}a_n^{L,\infty}=a^{L,\infty}.
\end{equation}
\end{proposition}

\begin{proof}[Proof of Proposition~\ref{prop:halfinfiniteconv}]
By the triangle inequality, one can write for any $L,M$ and $n$
\begin{equation}\label{triangle}
\left|a_n^{L,\infty}-a^{L,\infty}\right|\le\left|a_n^{L,\infty}-a_n^{L,M}\right|+\left|a_n^{L,M}-a^{L,M}\right|+\left|a^{L,M}-a^{L,\infty}\right|.
\end{equation}
By the uniform convergence in Proposition~\ref{prop:halfinfiniteFredholm} the first term on the right-hand side of \eqref{triangle} goes to $0$ uniformly in $n$ as $M\to\infty$.
Hence we can choose $M$ large enough so that the first and also the third term on the right-hand side of \eqref{triangle} are arbitrarily small.
With this $M$, we can use Theorem~\ref{thm:Ybelowacurve} on the interval $[L,M]$ to get that the second term on the right-hand side of \eqref{triangle} is small if $n$ is large enough,
from which we conclude \eqref{halfinfiniteconv}.
\end{proof}

In the next proposition, we bound the probability that the top curve in the tiling of the Aztec diamond hits $g_n$ in the interval $[L,\infty)$ uniformly in $n$.

\begin{proposition}\label{prop:hittingfar}
There are $c>0$ and $C\in\R$ such that for any $L>0$
\begin{equation}\label{hittingfar}
1-a_n^{L,\infty}\le Ce^{-cL}
\end{equation}
holds uniformly in $n$.
Similarly there are $c>0$ and $C\in\R$ so that for any $L>0$
\begin{equation}\label{hittingfarlim}
1-a^{L,\infty}\le Ce^{-cL}.
\end{equation}
\end{proposition}

\begin{proof}[Proof of Proposition~\ref{prop:hittingfar}]
By using the rescaled kernels introduced in \eqref{defKresc}, we have
\begin{equation}\label{aLinftyresc}
a_n^{L,\infty}=\det\left(\id-K_n^{L,\infty,\text{resc}}\right)_{L^2(\R_+)}
\end{equation}
for which kernel
\begin{equation}\label{KnLinftybound}
\left|K_n^{L,\infty,\text{resc}}(x,y)\right|\le Ce^{-c(x+y+L)}
\end{equation}
holds with some $c>0$ and $C\in\R$ by \eqref{kdecay} in Lemma~\ref{lemma:intkerneldecay}.
By \eqref{aLinftyresc}, the probability to be bounded is written as
\begin{equation}\begin{aligned}
1-\det\left(\id-K_n^{L,\infty,\text{resc}}\right)_{\ell^2(\Z_+)}&=\sum_{k=1}^\infty\frac1{k!}\int_{\R_+}\d x_1\dots\int_{\R_+}\d x_k\det\left(K_n^{L,\infty,\text{resc}}(x_l,x_m)\right)_{l,m=1}^k\\
&\le\sum_{k=1}^\infty\frac1{k!}\int_{\R_+}\d x_1\dots\int_{\R_+}\d x_kC^ke^{-2c(x_1+\dots+x_k)-ckL}k^{k/2}\\
&=\sum_{k=1}^\infty\frac{k^{k/2}}{k!}e^{-ckL}\left(\frac C{2c}\right)^k\\
&\le C'e^{-cL}
\end{aligned}\end{equation}
where we used the Fredholm expansion first, then \eqref{KnLinftybound} and Hadamard's inequality.
Now the uniform bound \eqref{hittingfar} follows.
The proof of \eqref{hittingfarlim} is similar based on the decay bound
\begin{equation}
\left|K^{L,\infty}(x,y)\right|\le Ce^{-c(x+y+L)}
\end{equation}
that can be deduced from the form of the kernel in \eqref{defKLM} and from \eqref{kappadecay} in Lemma~\ref{lemma:intkerneldecay}.
\end{proof}

\begin{proof}[Proof of Theorem~\ref{thm:YbelowacurveR}]
With the notation \eqref{defanLM}--\eqref{defaLM}, the statement of the theorem is equivalent to
\begin{equation}\label{limitwitha}
1-a_n^{-\infty,\infty}\to1-a^{-\infty,\infty}
\end{equation}
as $n\to\infty$. We give upper and lower bounds on the left-hand side of \eqref{limitwitha} as follows.
Since $1-a_n^{-\infty,\infty}$ is the probability that $Y(b_n(\tau))$ hits $g_n(\tau)$ for some $\tau\in\R$,
we obtain an upper bound for any $L>0$ by writing this event as the union of the events when hitting happens for $\tau\in(-\infty,-L]$, $\tau\in[-L,L]$ or $\tau\in[L,\infty)$ and by using union bound:
\begin{equation}\label{upper}
1-a_n^{-\infty,\infty}\le\left(1-a_n^{-\infty,-L}\right)+\left(1-a_n^{-L,L}\right)+\left(1-a_n^{L,\infty}\right).
\end{equation}
By \eqref{hittingfar}, the third and by symmetry the first term on the right-hand side of \eqref{upper} are uniformly small in $n$ if $L$ is large enough.
For any fixed $L>0$, by taking $n$ large, the second term on the right-hand side of \eqref{upper} is close to $1-a^{-L,L}$.
If $L$ was large enough, $1-a^{-L,L}$ is close enough to $1-a^{-\infty,\infty}$ by \eqref{hittingfarlim}.

A lower bound is obtained by the monotonicity of the events involved, i.e.
\begin{equation}\label{lower}
1-a_n^{-\infty,\infty}\ge1-a_n^{-L,L}
\end{equation}
for any $L>0$.
As $n\to\infty$, the lower bound converges to $1-a^{-L,L}$ which is close to $1-a^{\infty,\infty}$ if $L$ was chosen to be large enough.
The matching upper and lower bounds complete the proof of \eqref{limitwitha} and that of the theorem.
\end{proof}

\section{Identification of the hard-edge tacnode kernel}\label{s:identification}

This section is devoted to the proof of Theorems~\ref{thm:Xbelowh} and~\ref{thm:Xfindimdistr}.
Suppose that a function $g$ is given as in Theorem~\ref{thm:Xbelowh}.
Let us define the kernel
\begin{equation}\label{defKh}\begin{aligned}
K_g(x,y)&=\int_\R\d\xi\int_\R\d\zeta\Ai^{(t_1)}(R+x+\xi)\left(\phi_{t_2-t_1}(\xi,\zeta)-T_{t_1,t_2}^{g-R}(\xi,\zeta)\right)\Ai^{(-t_2)}(R+y+\zeta)\\
&\quad+\int_{\R_-}\d\xi\int_{\R_-}\d\zeta\Ai^{(t_1)}(R+x+\xi)T_{t_1,t_2}^{g-R}(\xi,\zeta)\Ai^{(-t_2)}(R+y-\zeta)\\
&\quad+\int_{\R_-}\d\xi\int_{\R_-}\d\zeta\Ai^{(t_1)}(R+x-\xi)T_{t_1,t_2}^{g-R}(\xi,\zeta)\Ai^{(-t_2)}(R+y+\zeta)\\
&\quad-\int_{\R_-}\d\xi\int_{\R_-}\d\zeta\Ai^{(t_1)}(R+x-\xi)T_{t_1,t_2}^{g-R}(\xi,\zeta)\Ai^{(-t_2)}(R+y-\zeta)
\end{aligned}\end{equation}
with the notations \eqref{defTg} and \eqref{defphi}.
We remark that the kernel $K_g$ can formally be defined by
\begin{equation}
K_g(x,y)=\id(x,y)-\int_{\R_-}\d u\int_{\R_-}\d v\,\Phi_{t_1}^x(u)T_{t_1,t_2}^{g-R}(u,v)\Psi_{t_2}^y(v)
\end{equation}
which involves the difference of two operators that are not trace class.

\begin{proposition}\label{prop:Aibelowh}
For any function $g$ as in Theorem~\ref{thm:Xbelowh}, we have
\begin{equation}\label{Aibelowh}
\P\left(\mathcal A_2(t)\le g(t)\textrm{ for all }t\in\R\right)=\det(\id-K_g)_{L^2(\R_+)}.
\end{equation}
\end{proposition}

\begin{lemma}\label{lem:appB}
Let $t_1<t_2$.
Then it holds
\begin{equation}\label{appB}
M_+^{t_1}(\xi,y)=\Ai^{(-t_1)}(\xi+y)-\int_\R\d\zeta\,T^g_{t_1,t_2}(\xi,\zeta)\Ai^{(-t_2)}(\zeta+y)+ (T^{g}_{t_1,t_2} M_+^{t_2})(\xi,y).
\end{equation}
\end{lemma}

\begin{proof}
Using \eqref{AiPhi} we get
\begin{equation}\begin{aligned}
M_+^{t_1}(\xi,y) &= \int_{t_1}^{t_2} \int_{-\infty}^R \P(T_+^{\xi,t_1}\in\d t,X_+^{\xi,T_1}\in\d\zeta)\int_\R\d v\,\phi_{t_2-t_1}(\zeta,v) \Ai^{(-t_2)}(v+y)\\
&\quad+\int_{t_2}^\infty \int_{-\infty}^R \P(T_+^{\xi,t_1}\in\d t,X_+^{\xi,t_1}\in\d\zeta)\Ai^{(-t)}(v+y).
\end{aligned}\end{equation}
The first term can be written as
\begin{equation}\begin{aligned}
&\int_\R\d v\,\P_{b(t_1)=\xi}(T_+^{\xi,t_1}\in [t_1,t_2], b(t_2)\in\d v) \Ai^{(-t_2)}(v+y)\\
&\qquad=\int_\R\d v\,\phi_{t_2-t_1}(\xi,v)\Ai^{(-t_2)}(v+y)-\int_\R\d v\,T^g_{t_1,t_2}(\xi,v)\Ai^{(-t_2)}(v+y)\\
&\qquad=\Ai^{(-t_1)}(\xi+y)-\int_{-\infty}^R\d v\,T^{g}_{t_1,t_2}(\xi,v)\Ai^{(-t_2)}(v+y).
\end{aligned}\end{equation}
The second one can be written, by decomposing with respect to the value of the Brownian bridge at time $t_2$, as
\begin{multline}
\int\d v\,T^g_{t_1,t_2}(\xi,v)\int_{t_2}^\infty \int_\R \P(T_+^{v,t_2}\in\d t,X_+^{v,t_2}\in\d\zeta) \Ai^{(-t)}(\zeta+y)\\
=\int_{-\infty}^R\d v\,T^g_{t_1,t_2}(\xi,v) M_+^{t_2}(v,y).
\end{multline}
\end{proof}

\begin{proposition}\label{prop:compatibility}
For any $t_1<t_2$ and $u,v\in\R_-$, the following compatibility relations are satisfied:
\begin{align}
\int_{\R_-}\d u\,\Phi_{t_1}^\xi(u)T^0_{t_1,t_2}(u,v)&=\Phi_{t_2}^\xi(v),\label{PhiT}\\
\int_{\R_-}\d v\,T^0_{t_1,t_2}(u,v)\Psi_{t_2}^\zeta(v)&=\Psi_{t_1}^\zeta(u),\label{TPsi}
\end{align}
where transition operator $T^0_{t_1,t_2}$ is the special case of \eqref{defTg} for the $g\equiv0$ function.
\end{proposition}

\begin{proof}[Proof of Proposition~\ref{prop:compatibility}]
By the reflection principle
\begin{equation}
T^0_{t_1,t_2}(x,y)=\phi_{t_2-t_1}(y-x)-\phi_{t_2-t_1}(y+x)
\end{equation}
with the notation \eqref{defphi}.
Hence the left-hand side of \eqref{PhiT} is equal to the sum of four integrals over $\R_-$ after plugging in the definition of $\Phi_{t_1}^\xi(u)$.
With a change of variables $u\to-u$ one turns two of them into integrals over $\R_+$ which can be combined with the remaining two to get two integrals over $\R$.
Then Lemma~\ref{lem:app} applies and proves \eqref{PhiT}.
The identity \eqref{TPsi} is seen similarly.
\end{proof}

\begin{proof}[Proof of Proposition~\ref{prop:Aibelowh}]
We first rewrite the statement of Lemma~\ref{lem:appB}.
Note that since $g=R$ on $[t_2,\infty)$, $M_+^{t_2}(\xi,y)=\Ai^{(-t_2)}(2R+y-\xi)$ for $\xi<R$ by \eqref{Mpmflat}.
Hence by using Lemma~\ref{lem:app}, we can write \eqref{appB} as
\begin{multline}\label{M+t1}
M_+^{t_1}(\xi,y)=\int_{-\infty}^R\d\zeta\left(\phi_{t_2-t_1}(\xi,\zeta)-T^g_{t_1,t_2}(\xi,\zeta)\right)\Ai^{(-t_2)}(\zeta+y)\\
+\int_{-\infty}^R\d\zeta\,T^{g}_{t_1,t_2}(\xi,\zeta)\Ai^{(-t_2)}(2R+y-\zeta).
\end{multline}
By \eqref{Mpmflat} again, we also have
\begin{equation}\label{M-t1}
M_-^{t_1}(x,\xi)=\Ai^{(t_1)}(2R+x-\xi)\id_{\xi<R}+\Ai^{(t_1)}(x+\xi)\id_{\xi\ge R}.
\end{equation}
Then by Theorem~\ref{thm:ReformulationKernel} with the choice $\alpha=t_1$, $L=-\infty$ and $M=\infty$ and by using \eqref{M-t1}
\begin{equation}\begin{aligned}
K_h(x,y)&=\int_{-\infty}^R\d\xi \Ai^{(t_1)}(2R+x-\xi)\Ai^{(-t_1)}(\xi+y)\\
&\quad+\int_R^\infty\d\xi \Ai^{(t_1)}(\xi+x)\Ai^{(-t_1)}(2R+y-\xi)\\
&\quad+\int_{-\infty}^R\d\xi\left(\Ai^{(t_1)}(\xi+x)-\Ai^{(t_1)}(2R+x-\xi)\right) M_+^{t_1}(\xi,y).
\end{aligned}\end{equation}
Direct computations yield \eqref{defKh} which involve the use of \eqref{M+t1}, Lemma~\ref{lem:app} and the shift of variables $\xi\to\xi-R$ and $\zeta\to\zeta-R$.
\end{proof}

Now we are ready to prove the two main theorems.
\begin{proof}[Proof of Theorem~\ref{thm:Xbelowh}]
By Proposition~\ref{prop:Aibelowh}, we need to prove the identity
\begin{equation}\label{detratio}
\frac{\det(\id-K_g)_{L^2(\R_+)}}{\det(\id-K_R)_{L^2(\R_+)}}=\det(\id-K_{t_1,t_1}+T^{g-R}_{t_1,t_2}K_{t_2,t_1})_{L^2(\R_-)}.
\end{equation}
We essentially follow the steps of the proof of Theorem 2.4 in~\cite{FV16}.
Consider a function $g$ that satisfies the conditions of Theorem~\ref{thm:Xbelowh}.
Observe that by \eqref{defKh}, the difference of kernels can be written as
\begin{equation}\label{kerneldiff}
K_g(x,y)-K_R(x,y)=\int_{\R_-}\d u\int_{\R_-}\d v\,\Phi_{t_1}^x(u)\left(T^0_{t_1,t_2}(u,v)-T_{t_1,t_2}^{g-R}(u,v)\right)\Psi_{t_2}^y(v).
\end{equation}
where $K_R$ is the kernel corresponding to $g=R$.

Let us recall that with $P_a$ we mean the projection on $[a,\infty)$ and with $\ol P_a$ the one on $(-\infty,a)$.
The ratio of Fredholm determinants on the left-hand side of \eqref{detratio} can be written as
\begin{equation}\label{detratiocompute}\begin{aligned}
\frac{\det(\id-K_g)_{L^2(\R_+)}}{\det(\id-K_R)_{L^2(\R_+)}}&=\det(\id-P_0(K_g-K_R)P_0(\id-K_R)^{-1}P_0)_{L^2(\R)}\\
&=\det(\id-P_0\Phi_{t_1}\ol P_0(T^0_{t_1,t_2}-T_{t_1,t_2}^{g-R})\ol P_0\Psi_{t_2}P_0(\id-K_R)^{-1}P_0)_{L^2(\R)}\\
&=\det(\id-(T^0_{t_1,t_2}-T_{t_1,t_2}^{g-R})\ol P_0\Psi_{t_2}P_0(\id-K_R)^{-1}P_0\Phi_{t_1})_{L^2(\R_-)}
\end{aligned}\end{equation}
where we used \eqref{kerneldiff} in the second equality above and the cyclic property of the determinant in the third.
By noting that on the right-hand side of \eqref{detratiocompute} the kernel $\Psi_{t_2}(\id-K_R)^{-1}\Phi_{t_1}=K_{t_2,t_1}$
and by the compatibility relation \eqref{TPsi}, the result \eqref{detratio} follows.
\end{proof}

\begin{proof}[Proof of Theorem~\ref{thm:Xfindimdistr}]
We apply Theorem~\ref{thm:Xbelowh} for the function
\begin{equation}\label{specg}
g(t)=\left\{\begin{array}{ll}u_l+t^2&\mbox{if }t=t_l, l\in\{1,\ldots,k\},\\R+t^2&\mbox{if }t\neq \{t_1,\ldots,t_k\}.\end{array}\right.
\end{equation}
Without loss of generality we may assume that $t_1<t_2<\dots<t_k$.
Theorem~\ref{thm:Xbelowh} implies that
\begin{equation}\label{Xbelowspecg}
\lim_{n\to\infty} \P\bigg(\bigcap_{\ell=1}^k \{X_n^{R,{\rm resc}}(t_\ell)\le u_\ell\}\bigg)=\det(\id-K_{t_1,t_1}+T^{g-R}_{t_1,t_k}K_{t_k,t_1})_{L^2(\R_-)}
\end{equation}
with the function given in \eqref{specg} on the right-hand side above.
Note that the special form of $g$ implies that the transition operator on the right-hand side of \eqref{Xbelowspecg} can be written as
\begin{equation}\label{Tg-Rt1tk}
T^{g-R}_{t_1,t_k}=\ol P_{u_1-R}T^0_{t_1,t_2}\ol P_{u_2-R}T^0_{t_2,t_3}\dots T^0_{t_{k-1},t_k}\ol P_{u_k-R}.
\end{equation}
Then the Fredholm determinant on the right-hand side of \eqref{Xbelowspecg} with $T^{g-R}_{t_1,t_k}$ replaced by the path integral kernel on the right-hand side of \eqref{Tg-Rt1tk}
is equal to the right-hand side of \eqref{Xfindimdistr} using Theorem 3.3 of~\cite{BCR13}.
To check the condition of the theorem, one sets $X=\R_-$ with operators $\ol Q_{t_i}=\ol P_{t_i}$, $\mathcal W_{t_i,t_j}=T^0_{t_i,t_j}$ and kernel $K_{t_i}=K_{t_i,t_i}$.
The compatibility assumptions follow from the form of the kernel $K_{t_i,t_j}=K^\ext(t_i,\cdot;t_j,\cdot)$ given in \eqref{defKext} and from Proposition~\ref{prop:compatibility}.
Boundedness and trace class property of certain operators are proved by observing that $Q_{t_i}$ is a projection to a finite interval.
However $K_{t_i}$ itself is not in trace class.
\end{proof}

\section{Proofs of asymptotics}\label{s:asympproofs}

As a preparation for the proof of Propositions~\ref{prop:pqconv},~\ref{prop:pbound},~\ref{prop:qbound} and~\ref{prop:improvedbound},
we rewrite the rescaled functions $p^{(n)}$ and $q^{(n)}$ using their definitions \eqref{defp}--\eqref{defq} and the scaling \eqref{scaling} as
\begin{equation}\label{pqrewrite}\begin{aligned}
P_l^{(n)}(v+j) &=2^{-5/6} n^{1/3}\frac{-1}{2\pi i}\oint_{\Gamma_1}\frac{\d z}ze^{-ng_0(z,0,0)+n^{2/3}tg_1(z)-n^{1/3}(y+\zeta)g_2(z)},\\
Q_{L}^{(n)}(u+i)&=2^{-5/6} n^{1/3}\frac1{2\pi i}\oint_{\Gamma_0}\frac{\d w}we^{ng_0(w,0,0)-n^{2/3}Lg_1(w)+n^{1/3}(\xi+x)g_2(w)}
\end{aligned}\end{equation}
where
\begin{multline}\label{defg0}
g_0(z,s,r)=\left(\frac1{2\sqrt2}-\frac12-s+r\right)\log z+\left(\frac12-\frac1{2\sqrt2}-s\right)\log(1-z)\\
+\left(\frac12+\frac1{2\sqrt2}+s\right)\log(1+z)+(r-2s)\log(\sqrt2+1)-\frac{\log2}2
\end{multline}
and
\begin{align}
g_1(z)&=2^{-7/6}\log\left(\frac{z(1-z)}{1+z}\right)+2^{-1/6}\log(\sqrt2+1),\label{defg1}\\
g_2(z)&=2^{-5/6}\log z+2^{-5/6}\log(\sqrt2+1).\label{defg2}
\end{align}

The integral formulas on the right-hand side of \eqref{pqrewrite} are useful when the rescaled variables $t$, $y$, $\xi+x$ remain bounded.
In order to understand the decay properties of $p^{(n)}$ and $q^{(n)}$, one has to bound them also when the above rescaled parameters are macroscopic
which corresponds to choosing the arguments $s$ and $r$ of $g_0(z,s,r)$ to be non-zero, since the asymptotics of the integrals in \eqref{pqrewrite}
are mainly determined by the behaviour of the function $g_0(z,s,r)$ along their integration contours.
First we keep the rescaled parameter $t$ bounded and consider the critical points of the corresponding function $g_0(z,0,r)$ which are
\begin{equation}\label{criticalpoints}
z_r^\pm=\frac{\sqrt2\pm\sqrt{8\sqrt2r+16r^2}}{2+\sqrt2+4r}.
\end{equation}
$z_r^\pm$ are two complex conjugate numbers for $r\in(-1/\sqrt2,0)$ and real otherwise.
We will pass through these critical points after the contours $\gamma$ and $\Gamma$ have been deformed.

\begin{lemma}\label{lemma:steep}
Fix $r\in\R$.
For $\theta\in[0,\pi]$, the function $\Re(g_0(w,0,r))$ decreases along the contour $w=w(\theta)=\rho e^{i\theta}$ as long as
\begin{equation}\label{steepwcond}
\cos\theta<\frac{1+\rho^2}{2\sqrt2\rho}
\end{equation}
holds and it increases if \eqref{steepwcond} does not hold.
For $\theta\in[\pi,2\pi]$, the opposite statement is true.

For $\phi\in[0,\pi]$, the function $-\Re(g_0(z,0,r))$ decreases along the contour $z=z(\phi)=1-Re^{-i\phi}$ as long as $Q>0$ holds with
\begin{equation}\label{defQ}
Q=6\sqrt2-4+(3\sqrt2+2)R^2+16r+4R^2r-8\sqrt2R(1+\sqrt2r)\cos\phi
\end{equation}
and it increases if $Q<0$ holds.
For $\phi\in[\pi,2\pi]$, the opposite statement is true.
\end{lemma}

\begin{proof}[Proof of Lemma~\ref{lemma:steep}]
By using the identity $\Re(\log z)=\frac12\log|z|^2$ for any complex number $z$, one gets for the contour $w=\rho e^{i\theta}$ that
\begin{equation}\label{reg0rewrite}\begin{aligned}
\Re(g_0(\rho e^{i\theta},0,r))
&=\left(\frac1{2\sqrt2}-\frac12+r\right)\log\rho+\left(\frac12-\frac1{2\sqrt2}\right)\frac12\log(1+\rho^2-2\rho\cos\theta)\\
&\quad+\left(\frac12+\frac1{2\sqrt2}\right)\frac12\log(1+\rho^2+2\rho\cos\theta)+r\log(\sqrt2+1)-\frac{\log2}2.
\end{aligned}\end{equation}
Its derivative simplifies to
\begin{equation}
\frac{\d}{\d\theta}\Re(g_0(\rho e^{i\theta},0,r))=-\frac{\rho\sin\theta\left(\sqrt2(1+\rho^2)-4\rho\cos\theta\right)}{2|1-w|^2|1+w|^2}.
\end{equation}
For $\theta\in[0,\pi]$, as long as the factor in parenthesis on the right-hand side above $\sqrt2(1+\rho^2)-4\rho\cos\theta>0$, i.e.\ if \eqref{steepwcond} holds,
then $\Re(g_0(\rho e^{i\theta},0,r))$ decreases, otherwise $\Re(g_0(\rho e^{i\theta},0,r))$ increases.

For the contour $z=1-Re^{-i\phi}$, similarly to \eqref{reg0rewrite}, one has
\begin{equation}\label{g0oncircle}\begin{aligned}
&-\Re(g_0(1-Re^{-i\phi},0,r))\\
&\qquad=-\left(\frac1{2\sqrt2}-\frac12+r\right)\frac12\log(1+R^2-2R\cos\phi)-\left(\frac12-\frac1{2\sqrt2}\right)\log R\\
&\qquad\quad-\left(\frac12+\frac1{2\sqrt2}\right)\frac12\log(4+R^2-4R\cos\phi)+r\log(\sqrt2+1)-\frac{\log2}2.
\end{aligned}\end{equation}
Its derivative can be calculated to be
\begin{equation}\label{Reg0derivative}
-\frac{\d}{\d\phi}\Re(g_0(1-Re^{-i\phi},0,r))=-\frac{RQ\sin\phi}{4|z|^2|1+z|^2}
\end{equation}
with $Q$ given in \eqref{defQ}.
For $\theta\in[0,\pi]$, if $Q>0$, then $-\Re(g_0(1-Re^{-i\phi},0,r))$ decreases, otherwise $-\Re(g_0(1-Re^{-i\phi},0,r))$ increases.
\end{proof}

\begin{proof}[Proof of Proposition~\ref{prop:pqconv}]
We choose the integration contours first.
Since $r=0$ in \eqref{pqrewrite}, the (double) critical point for $g_0(z,0,0)$ is at $z_0^\pm=\sqrt2-1$ by \eqref{criticalpoints}.
By the first part of Lemma~\ref{lemma:steep}, the function $\Re(g_0(w,0,0))$ is steep descent for the contour $w=\rho e^{i\theta}$ if $\rho\in(0,\sqrt2-1]$.
By the steep descent property, we mean that the function decreases along both arcs of the contour away from the critical point until the antipodal point.
The reason for the steep descent property of $\Re(g_0(w,0,0))$ is that the right-hand side of \eqref{steepwcond} is a decreasing function of $\rho$ and it is equal to $1$ for $\rho=\sqrt2-1$.
On the other hand, the function $-\Re(g_0(z,0,r))$ is steep descent for the contour $z=1-Re^{-i\phi}$ for $R\in(0,2-\sqrt2]$.
This is because $6\sqrt2-4+(3\sqrt2+2)R^2\ge8\sqrt2R$ for $R\in(0,2-\sqrt2]$, hence $Q\ge0$ for all values of $\phi$.

By Taylor expansion around $\sqrt2-1$, we obtain
\begin{equation}\label{gTaylor}\begin{aligned}
g_0(z,0,0)&=-\frac13\cdot2^{-5/2}(\sqrt2+1)^3(z-(\sqrt2-1))^3+\O((z-(\sqrt2-1))^4),\\
g_1(z)&=-2^{-5/3}(\sqrt2+1)^2(z-(\sqrt2-1))^2+\O((z-(\sqrt2-1))^3),\\
g_2(z)&=2^{-5/6}(\sqrt2+1)(z-(\sqrt2-1))+\O(z-(\sqrt2-1)^2).
\end{aligned}\end{equation}
The deformed integration contours are the following.
In a small neighbourhood of the critical point, we follow the direction of steepest descent, i.e.\ we use
\begin{equation}
\Gamma^{\delta,\pm}=\{e^{\pm i\pi/3}t:t\in[0,\delta]\},\qquad\gamma^{\delta,\pm}=\{e^{\pm2i\pi/3}t:t\in[0,\delta]\}.
\end{equation}
Let $\delta>0$ be small.
We deform the contour $\Gamma_1$ to $\Gamma^{\delta,\pm}$ completed by a circle around $1$ with radius slightly smaller than $2-\sqrt2$.
The function $-\Re(g_0(w,0,0))$ is steep descent along the $\Gamma^{\delta,\pm}$ part of the contour by the Taylor expansion \eqref{gTaylor}
and by a previous observation in this proof along the circular part.
Similarly, $\Gamma_0$ can be replaced $\gamma^{\delta,\pm}$ completed by a circle around $0$ with a radius slightly smaller than $\sqrt2-1$.
Then the function $\Re(g_0(w,0,0))$ is steep descent along the $\gamma^{\delta,\pm}$ by the Taylor expansion \eqref{gTaylor} and by a previous observation along the circular part.
By the Taylor expansion \eqref{gTaylor} again, the circular parts of the contours can be omitted by making an error of order $e^{-c\delta^3n}$.

By writing the Taylor expansion \eqref{gTaylor} to the right-hand side of \eqref{pqrewrite} and after the change of variables
$Z=2^{-5/6}(\sqrt2+1)(z-(\sqrt2-1))n^{1/3}$ and $W=2^{-5/6}(\sqrt2+1)(w-(\sqrt2-1))n^{1/3}$, we get
\begin{equation}\label{pZint}\begin{aligned}
P_l^{(n)}(v+j)&=\frac1{2\pi i}\int_{e^{-i\pi/3}\delta n^{1/3}}^{e^{i\pi/3}\delta n^{1/3}}\d Ze^{Z^3/3-tZ^2-(y+\zeta)Z}\\
&\quad+\frac1{2\pi i}\int_{e^{-i\pi/3}\delta n^{1/3}}^{e^{i\pi/3}\delta n^{1/3}}\d Ze^{Z^3/3-tZ^2-(y+\zeta)Z}\left(e^{n^{-1/3}\O(Z^4+Z^2)}-1\right),
\end{aligned}\end{equation}
and
\begin{equation}\label{qWint}\begin{aligned}
Q_{L}^{(n)}(u+i)&=\frac1{2\pi i}\int_{e^{-2i\pi/3}\delta n^{1/3}}^{e^{2i\pi/3}\delta n^{1/3}}\d We^{-W^3/3+LW^2+(\xi+x)W}\\
&\quad+\frac1{2\pi i}\int_{e^{-2i\pi/3}\delta n^{1/3}}^{e^{2i\pi/3}\delta n^{1/3}}\d We^{-W^3/3+LW^2+(\xi+x)W}\left(e^{n^{-1/3}\O(W^4+W^2)}-1\right).
\end{aligned}\end{equation}
Using the bound $|e^x-1|\le|x|e^{|x|}$, one can see that the last term on the right-hand side of \eqref{pZint} can be bounded as
\begin{multline}
\left|\frac1{2\pi i}\int_{e^{-i\pi/3}\delta n^{1/3}}^{e^{i\pi/3}\delta n^{1/3}}\d Ze^{Z^3/3-tZ^2-(y+\zeta)Z}\left(e^{n^{-1/3}\O(Z^4+Z^2)}-1\right)\right|\\
\le\left|\frac{n^{-1/3}}{2\pi i}\int_{e^{-i\pi/3}\delta n^{1/3}}^{e^{i\pi/3}\delta n^{1/3}}\d Z\O(Z^4+Z^2)e^{Z^3/3-tZ^2-(y+\zeta)Z+n^{-1/3}\O(Z^4+Z^2)}\right|
\end{multline}
which is integrable and goes to $0$ as $n^{-1/3}$.
The last error term in \eqref{qWint} can be bounded similarly.
One can extend the integration path to infinity in the first term on the right-hand side of \eqref{pZint} and \eqref{qWint} by making an error of order $e^{-c\delta^3 n}$.
Computing the respective Airy integrals yield \eqref{pconv} and \eqref{qconv}.
The bound on the error terms above are uniform on compact intervals of $\zeta+y$ and $x+\xi$, hence the uniformity of the convergence follows.
\end{proof}

\begin{proof}[Proof of Proposition~\ref{prop:pbound}]
Assume first that $y+\zeta>0$ and let $y+\zeta=2^{5/6}rn^{2/3}$ where $r>0$ is a macroscopic parameter for which $j+2^{-5/6}\zeta n^{1/3}=rn$.
Then instead of the first equation in \eqref{pqrewrite}, we write
\begin{equation}\label{prewrite}
P_l^{(n)}(v+j)=2^{-5/6} n^{1/3}\frac{-1}{2\pi i}\oint_{\Gamma_1}\frac{\d z}z\,e^{-ng_0(z,0,r)+n^{2/3}tg_1(z)}.
\end{equation}
The function $g_0(z,0,r)$ has two critical points given in \eqref{criticalpoints} which are $\O(r^{1/2})$ away from $\sqrt2-1$ for small $r>0$.
We pass by the one above $\sqrt2-1$ where $-\Re(g_0(z,0,r))$ is smaller.

For a small $\varepsilon>0$, we choose
\begin{equation}\label{defalpha}
\alpha=\left\{\begin{array}{ll} \sqrt2-1+r^{1/2} & \mbox{if }r\le\varepsilon,\\ \sqrt2-1+\varepsilon^{1/2} & \mbox{if }r>\varepsilon.\end{array}\right.
\end{equation}
Next we show that the function $-\Re(g_0(z,0,r))$ is steep descent along the contour $z(\phi)=1-(1-\alpha)e^{-i\phi}$ for any $r>0$ if $\varepsilon$ is small enough,
i.e.\ it attains its maximum at $\alpha$ and it decreases along both arcs until the point $2-\alpha$.
To this end, we have to show that the quantity $Q$ defined in \eqref{defQ} is positive along the contour.
First note that $\phi\mapsto Q$ is increasing for $\phi\in[0,\pi]$, hence it is enough to prove positivity for $\phi=0$.
Further, $Q(\phi=0)$ is a linear function of $r$, that is,
\begin{equation}
Q(\phi=0)=6\sqrt2-4-8\sqrt2R+(3\sqrt2+2)R^2+4(R-2)^2r.
\end{equation}
The constant term $6\sqrt2-4-8\sqrt2R+(3\sqrt2+2)R^2$ is positive for any $R<2-\sqrt2$, hence it is positive for $R=1-\alpha$ with $\alpha$ defined in \eqref{defalpha}.
This proves that $Q(\phi=0)>0$ for any $r>0$ and that the quantity $Q>0$ along the contour $z(\phi)=1-(1-\alpha)e^{-i\phi}$, i.e.\ the steep descent property holds.

By the steep descent property, we can localize the integral
\begin{multline}\label{plocalize}
\big|P_l^{(n)}(v+j)\big|=2^{-5/6} n^{1/3}e^{\Re(-ng_0(\alpha,0,r)+n^{2/3}tg_1(\alpha))}\\
\times\left(\left|\frac1{2\pi i}\int_{-\delta}^\delta\d\phi\,(1-\alpha)e^{n(-g_0(z(\phi),0,r)+g_0(\alpha,0,r))+n^{2/3}t(g_1(z(\phi))-g_1(\alpha))}\right|+\O(e^{-\wt cn})\right)
\end{multline}
with a constant $\wt c>0$ independent of $n$ where $z(\phi)=1-(1-\alpha)e^{-i\phi}$.
Series expansion yields
\begin{equation}
-\Re(g_0(z(\phi),0,r)-g_0(\alpha,0,r))=-\gamma\phi^2+\O(\phi^3)
\end{equation}
with
\begin{equation}
\gamma=(1-\alpha)\left(\frac{4+2\sqrt2}{(1+\alpha)^2}-\frac{2-\sqrt2}{\alpha^2}\right)=(4+2\sqrt2)(\alpha-(\sqrt2-1))+\O((\alpha-(\sqrt2-1))^2)
\end{equation}
where the second equality is the expansion of $\gamma$ for $\alpha$ close to $\sqrt2-1$.
In particular $\gamma>0$ if $\alpha$ is slightly larger than $\sqrt2-1$.
In addition, $\Re(g_1(z(\phi))-g_1(\alpha))$ is also quadratic in $\phi$.
Hence
\begin{multline}\label{plocalcompute}
\left|\frac1{2\pi i}\int_{-\delta}^\delta\d\phi\,e^{n(-g_0(z(\phi),0,r)+g_0(\alpha,0,r))+n^{2/3}t(g_1(z(\phi))-g_1(\alpha))}\right|\\
=\frac1{2\pi}\int_{-\delta}^\delta\d\phi\,e^{-\gamma\phi^2n(1+\O(\phi))(1+\O(n^{-1/3}))}
\le\frac1{2\pi}\int_{-\delta}^\delta\d\phi\,e^{-\frac{\gamma\phi^2n}2}
\le\frac1{\sqrt{2\pi n\gamma}}
\end{multline}
for $n$ large enough and $\delta$ small enough.
The last inequality above follows by computing the Gaussian integral over $\R$.
Note that one can write
\begin{equation}\label{1/sqrt2pingamma}
\frac{n^{1/3}}{\sqrt{2\pi n\gamma}}=\frac1{\sqrt{2\pi}}\left\{\begin{array}{ll} n^{-1/6}r^{-1/4} & \mbox{if }r\le\varepsilon,\\ n^{-1/6}\varepsilon^{-1/4} & \mbox{if }r>\varepsilon.\end{array}\right.
\end{equation}
where $n^{-1/6}r^{-1/4}=(y+\zeta)^{-1/4}$, hence \eqref{1/sqrt2pingamma} is bounded if $y+\zeta$ is at least a positive constant.
By putting together \eqref{plocalize}, \eqref{plocalcompute} and \eqref{1/sqrt2pingamma}, one can conclude that for $n$ large enough
\begin{equation}\label{pboundatalpha}
\big|P_l^{(n)}(v+j)\big|\le 2^{-5/6} n^{1/3}Ce^{\Re(-ng_0(\alpha,0,r)+n^{2/3}tg_1(\alpha))}
\end{equation}
uniformly if $y+\zeta$ is at least a positive constant.

By Taylor expansion,
\begin{multline}\label{gz0rTaylor}
g_0(z,0,r)=\left(-\frac{2^{-5/2}}3(\sqrt2+1)^3(z-(\sqrt2-1))^3+(\sqrt2+1)r(z-(\sqrt2-1))\right)\\
\times(1+\O(z-(\sqrt2-1))).
\end{multline}
If $r\le\varepsilon$, then $\alpha-(\sqrt2-1)=r^{1/2}$ and by \eqref{pboundatalpha} and using \eqref{gTaylor} and \eqref{gz0rTaylor},
\begin{equation}\label{pbounds<eps}\begin{aligned}
\big|P_l^{(n)}(v+j)\big|&\le C e^{\left(-\frac{14+17\sqrt2}{24}nr^{3/2}-2^{-5/3}(\sqrt2+1)^2tn^{2/3}r\right)(1+\O(\sqrt\varepsilon))}\\
&=Ce^{\left(-\frac{14+17\sqrt2}{24}(y+\zeta)^{3/2}-2^{-5/3}(\sqrt2+1)^2t(y+\zeta)\right)(1+\O(\sqrt\varepsilon))}
\end{aligned}\end{equation}
where $-\frac{14+17\sqrt2}{24}=\frac13\cdot2^{-5/2}(\sqrt2+1)^3-(\sqrt2+1)$.
The first term in the exponent on the right-hand side of \eqref{pbounds<eps} dominates,
hence for any $c>0$, one can choose $C'$ large enough so that $Ce^{-c(y+\zeta)}$ upper bounds the right-hand side of \eqref{pbounds<eps} if $y+\zeta$ is at least a positive constant.

If $r>\varepsilon$, then
\begin{equation}\label{pboundintermediate}
\big|P_l^{(n)}(v+j)\big|\le C e^{\left(n\sqrt\varepsilon\left(\frac13\cdot2^{-5/2}(\sqrt2+1)^3\varepsilon-r\right)-2^{-5/3}(\sqrt2+1)^2tn^{2/3}\varepsilon\right)(1+\O(\sqrt\varepsilon))}
\end{equation}
where $\frac13\cdot2^{-5/2}(\sqrt2+1)^3\varepsilon-r\le-\frac1{10}r$ if $\varepsilon$ is not too large.
Hence the first term in the exponent on the right-hand side of \eqref{pboundintermediate} can be upper bounded by $-\frac1{10}\sqrt\varepsilon n^{1/3}(y+\zeta)$
and the second term can simply be neglected since it is negative.
That is, for given $\varepsilon>0$, $n$ can be chosen so large that
\begin{equation}\label{pbounds>eps}
\big|P_l^{(n)}(v+j)\big|\le Ce^{-\frac1{20}\sqrt\varepsilon n^{1/3}(y+\zeta)}.
\end{equation}
This finishes the proof for the case when $y+\zeta$ is bounded from below by a positive constant.
By the uniform convergence in Proposition~\ref{prop:pqconv}, \eqref{pbound} can be extended for the case when $y+\zeta$ is bounded from below by any constant.
\end{proof}

\begin{proof}[Proof of Proposition~\ref{prop:qbound}]
First we prove \eqref{qbound} for $x+\xi>0$.
In the case when $x+\xi=2^{5/6}rn^{2/3}$ is macroscopic with some $r>0$, we can write
\begin{equation}\label{qrewrite}
Q_{L}^{(n)}(u+i)=2^{-5/6}n^{1/3}\frac1{2\pi i}\oint_{\Gamma_0}\frac{\d w}w\,e^{ng_0(w,0,r)-n^{2/3}Lg_1(w)}
\end{equation}
similarly to \eqref{prewrite}.
By contour deformation, we pass by the critical point below $\sqrt2-1$, i.e.\ for a fixed $\varepsilon>0$ small enough, we choose
\begin{equation}
\wh\alpha=\left\{\begin{array}{ll} \sqrt2-1-r^{1/2} & \mbox{if }r\le\varepsilon,\\ \sqrt2-1-\varepsilon^{1/2} & \mbox{if }r>\varepsilon.\end{array}\right.
\end{equation}
Since the right-hand side of \eqref{steepwcond} is larger than $1$ for any $\rho\in(0,\sqrt2-1)$,
the function $\Re(g_0(w,0,r))$ is steep descent along the contour $w(\theta)=\wh\alpha e^{i\theta}$ for any $r>0$,
i.e.\ it attains its maximum at $\wh\alpha$ and it decreases along both arcs until the point $-\wh\alpha$.
Hence one can localize the integration contour to a small $\delta>0$ neighbourhood of $\wh\alpha$ in the same way as in \eqref{plocalize}.

By series expansion,
\begin{equation}
\Re(g_0(\wh\alpha e^{i\theta},0,r)-g_0(\wh\alpha,0,r))=-\wh\gamma\theta^2+\O(\theta^3)
\end{equation}
with
\begin{equation}
\wh\gamma=\frac{\sqrt2\wh\alpha(\wh\alpha^2-2\sqrt2\wh\alpha+1)}{4(1-\wh\alpha^2)}=-\frac{2+\sqrt2}8(\wh\alpha-(\sqrt2-1))+\O((\wh\alpha-(\sqrt2-1))^2)
\end{equation}
where the second equality above is the expansion of $\wh\gamma$ for $\wh\alpha$ close to $\sqrt2-1$.
In particular, $\wh\gamma>0$ if $\wh\alpha\in(0,\sqrt2-1)$.
Further $\Re(g_1(\wh\alpha e^{i\theta})-g_1(\wh\alpha))$ is also quadratic in $\theta$ for small $\theta$.
The rest of the proof of the bound \eqref{qbound} for $x+\xi>0$ on the rescaled $q^{(n)}$ can be done analogously to the one for $p^{(n)}$, in particular
\begin{equation}\label{qboundatalpha}
\big|Q_{L}^{(n)}(u+i)\big|\le\frac{Cn^{1/3}e^{\Re(ng_0(\wh\alpha,0,r)+n^{2/3}g_2(\wh\alpha))}}{\sqrt{2\pi n\wh\gamma}}
\end{equation}
can be shown.
The asymptotics of the first factor on the right-hand side of \eqref{qboundatalpha} is the same as \eqref{1/sqrt2pingamma}
and the exponential factor can be bounded as in \eqref{pbounds<eps}--\eqref{pbounds>eps}.
This yields the bound \eqref{qbound} for $x+\xi>0$.

Next we consider the remaining cases, i.e.\ when $x+\xi<0$.
Let $x+\xi=2^{5/6}rn^{2/3}$ again where $r<0$ for the rest of this proof.
We show
\begin{equation}\label{qboundnegative}
\big|Q_{L}^{(n)}(u+i)\big|\le C
\end{equation}
when $r\in(-1/\sqrt2,-\varepsilon)$ for some small $\varepsilon>0$ which is a stronger bound than \eqref{qbound} in this regime.
For any $r\in(-1/\sqrt2,0)$, the function $g_0(w,0,r)$ has exactly two complex conjugate critical points given by \eqref{criticalpoints}.
Let us choose the integration contours to be the unique circle of the form $w=w(\theta)=\rho e^{i\theta}$ which passes through the critical point $z_r^\pm$, i.e.\ $\rho=|z_r^\pm|$.
By Lemma~\ref{lemma:steep}, along the contour $w=\rho e^{i\theta}$,
the function $\Re(g_0(w,0,r))$ attains its maximum for $\theta\in[0,\pi]$ exactly at the critical point for the unique $\theta$ for which equality holds in \eqref{steepwcond}.

Then by \eqref{qrewrite}
\begin{equation}\label{qboundbulk}
\big|Q_{L}^{(n)}(u+i)\big|\le2^{-5/6}n^{1/3}e^{n\Re(g_0(z_r^+,0,r))}\int_0^{2\pi}\frac{\d\theta}{2\pi}\rho e^{n\Re(g_0(\rho e^{i\theta},0,r)-g_0(z_r^+,0,r))+n^{2/3}L\Re(g_1(\rho e^{i\theta}))}.
\end{equation}
By the maximum property of $\Re(g_0(w,0,r))$ along the contour $w=\rho e^{i\theta}$, the leading term in the exponent of the integrand above satisfies
\begin{equation}
\Re(g_0(\rho e^{i\theta},0,r)-g_0(z_r^+,0,r))\le0.
\end{equation}
Therefore, the integrand on the right-hand side of \eqref{qboundbulk} can grow at most as the exponential of $n^{2/3}$ times the maximum of $|L\Re(g_1)|$ along the contour.
The latter maximum is uniformly bounded except for the case when the contour passes by a singularity of $g_1$ at $\pm1$ or $0$.
In these cases the the circular contour can locally be modified to have a uniformly positive distance from the singularities and also to keep the maximum property of $\Re(g_0(w,0,r))$.
In this way, the exponent of the integrand on the right-hand side of \eqref{qboundbulk} is at least constant times $n^{2/3}$ which together with the $n^{1/3}$ prefactor
are dominated by the exponential prefactor provided that $\Re(g_0(z_r^+,0,r))<0$ for $r\in(-1/\sqrt2,0)$.
This yields boundedness in \eqref{qboundnegative} for $r\in(-1/\sqrt2,-\varepsilon)$.

To show the negativity of $\Re(g_0(z_r^+,0,r))$, we first observe that $\Re(g_0(z_0^+,0,0))=0$ and that $\frac{\d}{\d r}\Re(g_0(z_r^+,0,r))=0$.
The rest of the proof is showing that $\Re(g_0(z_r^+,0,r))$ is a concave function of $r$ in $[-1/\sqrt2,0]$.
By substituting the general formula \eqref{criticalpoints} for $z_r^+$ into \eqref{defg0}, one observes that the second derivative does not contain any logarithm and it simplifies to
\begin{equation}\label{g0second}
\frac{\d^2}{\d r^2}g_0(z_r^+,0,r)=-\frac{2i}{\sqrt{-r(\sqrt2+2r)}\left(\sqrt2+4r+4i\sqrt{-r(\sqrt2+2r)}\right)}.
\end{equation}
Since the quantities under the square root are non-negative for $r\in[-1/\sqrt2,0]$, one readily gets the real part of the two sides of \eqref{g0second} to be
\begin{equation}
\frac{\d^2}{\d r^2}\Re(g_0(z_r^+,0,r))=-\frac4{1-4\sqrt2r-8r^2}.
\end{equation}
This proves concavity because $1-4\sqrt2r-8r^2>0$ for $r\in[-1/\sqrt2,0]$, hence the bound \eqref{qboundnegative} follows for $r\in(-1/\sqrt2,-\varepsilon)$.

For $r\le-1/\sqrt2$, we choose the contour $w=\rho e^{i\theta}$ with $\rho=\sqrt2+1$ for which the function $\Re(g_0(w,0,r))$ is steep descent, since the right-hand side of \eqref{steepwcond} is equal to $1$.
After a similar bound as \eqref{qboundbulk} in the previous regime and by the same argument, it is enough to see that $\Re(g_0(\sqrt2+1,0,r))<0$ for $r\le-1/\sqrt2$.
The negativity holds for $r=-1/\sqrt2$ and by definition \eqref{defg0}, we have
\begin{equation}
g_0(\sqrt2+1,0,r)=g_0(\sqrt2+1,0,-1/\sqrt2)+2(r+1/\sqrt2)\log(\sqrt2+1),
\end{equation}
which proves the negativity of $\Re(g_0(\sqrt2+1,0,r))$ for $r\le-1/\sqrt2$ and the boundedness in \eqref{qboundnegative} in this regime.

Finally, let us consider the case when $r\in(-\varepsilon,0)$.
We show in what follows that there are $c>0$ and $C\in\R$ such that
\begin{equation}\label{intermediatebound}
\big|Q_{L}^{(n)}(u+i)\big|\le\frac C{|rn^{2/3}|^{1/4}}e^{-2nr^2+cn^{2/3}|r|}
\end{equation}
holds.
It is enough to conclude the proof of Proposition~\ref{prop:qbound} for the following reason.
If $|r|>cn^{-1/3}$, i.e.\ $x+\xi<-cn^{1/3}$ with the $c$ from \eqref{intermediatebound}, then the $-2nr^2$ term dominates in the exponent on the right-hand side of \eqref{intermediatebound},
hence the right-hand side can be bounded by a constant.
If $|r|<n^{-2/3}$, i.e.\ $x+\xi$ is of constant order, then the convergence \eqref{qconv} can be used instead to conclude that \eqref{qbound} holds.
In the intermediate case, we neglect the first term in the exponent on the right-hand side of \eqref{intermediatebound}
and the second term gives a $-c(x+\xi)$ in the exponent which together with the $C/|x+\xi|^{1/4}$ prefactor yield \eqref{qbound}.

To prove \eqref{intermediatebound}, we replace the integration contour $\gamma$ in \eqref{prewrite} by the union of the local paths
\begin{equation}
\gamma_{\rm loc}^\pm=\{z_r^\pm+e^{\pm i3\pi/4}x,x\in[-\Im(z_r^+)\sqrt2,\delta]\}
\end{equation}
for a small $\delta>0$ and the circular arc around the origin that connects the endpoints of $\gamma_{\rm loc}^\pm$.
The paths $\gamma_{\rm loc}^\pm$ intersect at $x=-\Im(z_r^+)\sqrt2$ on the real axis.
By the Taylor expansion also explained below in \eqref{g0taylor}--\eqref{g0derivatives}, the function $\Re(g_0(w,0,r))$ attains its two maxima along $\gamma_{\rm loc}^\pm$ at $z_r^\pm$.
The value of $\Re(g_0(w,0,r))$ further decreases along the circular part of the contour by Lemma~\ref{lemma:steep} for the following reason.
Let $\theta_r^+=\arg z_r^+$.
Then $\Re(g_0(|z_r^+|e^{i\theta},0,r))$ decreases for $\theta\in[\theta_r^+,\pi]$, in particular \eqref{steepwcond} holds for $\theta\in[\theta_r^+,\pi]$.
The radius of the circular part of the new contour is smaller than $|z_r^\pm|$, hence \eqref{steepwcond} and the decreasing property remain valid.

Next one localizes the integral to $\gamma_{\rm loc}^\pm$ by making an additive error of order $e^{n\Re(g_0(z_r^+,0,r))+\O(n\delta^3)}$.
To bound the integral on $\gamma_{\rm loc}^+$ (and similarly for $\gamma_{\rm loc}^-$), we use Taylor expansion around the critical point $z_r^+$
\begin{equation}\label{g0taylor}
g_0(w,0,r)=g_0(z_r^+,0,r)+\frac12g''_0(z_r^+,0,r)(w-z_r^+)^2+\frac16g'''_0(z_r^+,0,r)(w-z_r^+)^3+\O((w-z_r^+)^4)
\end{equation}
where primes mean derivatives in the first variable.
For $w\in\gamma_{\rm loc}^+$, $w-z_r^+$ has an angle $\pm e^{i3\pi/4}$, furthermore,
\begin{equation}\label{g''asymptotics}
g''_0(z_r^+,0,r)=-i2^{-1/4}(\sqrt2+1)^2\sqrt{|r|}+\O(|r|),\quad g'''_0(z_r^+,0,r)=-2^{-3/2}(\sqrt2+1)^3+\O(\sqrt{|r|})
\end{equation}
as $r\uparrow0$, hence with $w=z_r^+ + e^{i 3\pi/4}x$,
\begin{equation}\label{g0derivatives}\begin{aligned}
\Re\Big(\tfrac12g''_0(z_r^+,0,r)(w-z_r^+)^2\Big)&=-\tfrac12|g''_0(z_r^+,0,r)|x^2(1+\O(\sqrt{\varepsilon_0})),\\
\Re\Big(\tfrac16g'''_0(z_r^+,0,r)(w-z_r^+)^3\Big)&=-\tfrac1{6\sqrt2}|g'''_0(z_r^+,0,r)|x^3(1+\O(\sqrt{\varepsilon_0})).
\end{aligned}\end{equation}

The contribution that comes from the integral over $\gamma_{\rm loc}^+$ is bounded by
\begin{multline}\label{gammaloc+bound}
\bigg|\int_{\gamma_{\rm loc}^+}\frac{\d w}w\,e^{ng_0(w,0,r)+n^{2/3}Lg_1(w)}\bigg|\le Ce^{\Re(ng_0(z_r^+,0,r))+n^{2/3}L\sup_{w\in\gamma_{\rm loc}}\Re(g_1(w))}\\
\times\int_{-\Im(z_r^+)\sqrt2}^\delta\d x\,e^{-\frac n2|g''_0(z_r^+,0,r)|x^2-\frac n{6\sqrt2}|g'''_0(z_r^+,0,r)|x^3+\O(nx^4)}.
\end{multline}

For $-\Im(z_r^+)\sqrt2\le x\le0$, after comparing the numerical values of the derivatives, one can dominate the cubic term by the quadratic one
\begin{equation}\label{quadratictermbound}
-\frac n2|g''_0(z_r^+,0,r)|x^2-\frac n{6\sqrt2}|g'''_0(z_r^+,0,r)|x^3\le-\frac n4|g''_0(z_r^+,0,r)|x^2
\end{equation}
if $n$ is large enough.
By replacing the factor $n/4$ by $n/6$, the quartic error term $\O(nx^4)$ can also be dominated if $\varepsilon_0$ is small enough.
On the other hand, for $0\le x\le\delta$, the cubic term is negative and it dominates the error term, i.e.\ for $\delta$ small enough,
\begin{equation}\label{cubictermbound}
-\frac n{6\sqrt2}|g'''_0(z_r^+,0,r)|x^3+\O(nx^4)\le-\frac n{12\sqrt2}|g'''_0(z_r^+,0,r)|x^3\le0
\end{equation}
holds.

By combining the previous bounds, the integral on the right-hand side of \eqref{gammaloc+bound} can be bounded by
\begin{multline}\label{gammalocintegral}
\int_{-\Im(z_r^+)\sqrt2}^\delta\d x\,e^{-\frac n2|g''_0(z_r^+,0,r)|x^2-\frac n{6\sqrt2}|g'''_0(z_r^+,0,r)|x^3+\O(nx^4)}\\
\le\int_{-\Im(z_r^+)\sqrt2}^\delta\d x\,e^{-\frac n6|g''_0(z_r^+,0,r)|x^2}
\le\sqrt{\frac{6\pi}{n|g_0''(z_r^+,0,r)|}}
=\frac C{|r|^{1/4}\sqrt n}
\end{multline}
where we extended the integral to $\R$ and computed the Gaussian integral in the second inequality, where we used the asymptotics \eqref{g''asymptotics} as well.
Since
\begin{equation}\label{g1asymptotics}
\Re(g_1(z_r^+))=-2^{5/6}r+\O(r^2),\qquad\Re(g_1'(z_r^+))=-(90+58\sqrt2)^{1/3}r+\O(r^2),
\end{equation}
also the supremum $\sup_{w\in\gamma_{\rm loc}}\Re(g_1(w))=\O(|r|)$.
On the other hand,
\begin{equation}\label{g0asymptotics}
\Re(g_0(z_r^+,0,r))=-2r^2+\O(|r|^3).
\end{equation}
Putting together \eqref{gammaloc+bound} and \eqref{gammalocintegral} with \eqref{g1asymptotics} and \eqref{g0asymptotics}, the bound \eqref{intermediatebound} follows.
\end{proof}

\begin{proof}[Proof of Proposition~\ref{prop:improvedbound}]
Suppose that the parameters $t=2^{7/6}sn^{1/3}$ and $y+\zeta=2^{5/6}rn^{2/3}$ are macroscopic where $s,r>0$.
With this setting of parameters, one has the representation
\begin{equation}\label{pimprovedrepr}
P_l^{(n)}(v+j)=2^{-5/6}n^{1/3}\frac{-1}{2\pi i}\oint_{\Gamma_1}\frac{\d z}z\,e^{-ng_0(z,s,r)}
\end{equation}
which can be checked by comparing \eqref{pqrewrite} with \eqref{defg0}--\eqref{defg2}.
Note that if $s\ge\frac12-\frac1{2\sqrt2}$, then the integrand on the right-hand side of \eqref{pimprovedrepr} has no singularity at $1$ and hence inside $\Gamma_1$ and then the whole integral is zero.
Therefore it is enough to consider $s\in(0,\frac12-\frac1{2\sqrt2})$ in the rest of the proof.

Let us deform the integration contour $\Gamma_1$ in \eqref{pimprovedrepr} first.
For $r=0$, the function $g_0(z,s,0)$ has critical points at $z_1=\sqrt2-1$ and at $z_2=\frac{\sqrt2+4s+4\sqrt2s}{2+\sqrt2-4s}$.
For $s\in(0,\frac12-\frac1{2\sqrt2})$, one has $z_2\in(\sqrt2-1,1)$ and let the integration contour be the circle around $1$ which passes through $z_2$, i.e.\ with radius $R=1-z_2$.
Then one can write $-\Re(g_0(1-Re^{-i\phi},s,0))$ analogously to \eqref{g0oncircle} and by taking its derivative one arrives to
\begin{equation}\label{Reg0derivatives}
-\frac{\d}{\d\phi}\Re(g_0(1-Re^{-i\phi},s,0))=-\frac{R\wt Q\sin\phi}{4|z|^2|1+z|^2}
\end{equation}
where $z=1-Re^{-i\phi}$ and $\wt Q=6\sqrt2-4+(3\sqrt2+2)R^2-8s+4R^2s-8\sqrt2R\cos\phi$.
Since the function $\phi\mapsto\wt Q$ is increasing, it is enough to show that it is positive for $\phi=0$
in order to verify the steep descent property of the function $-\Re(g_0(z,s,0))$ along the contour $1-Re^{-i\phi}$ as in Lemma~\ref{lemma:steep}.
The quantity $\wt Q$ with $\phi=0$ and with $R=1-z_2=1-\frac{\sqrt2+4s+4\sqrt2s}{2+\sqrt2-4s}$ is equal to
\begin{equation}
\wt Q\bigg(\phi=0,R=1-\frac{\sqrt2+4s+4\sqrt2s}{2+\sqrt2-4s}\bigg)=\frac{32(\sqrt2+1)s(1+8s+8s^2)}{(2+\sqrt2-4s)^2}
\end{equation}
which is positive for $s\in(0,\frac12-\frac1{2\sqrt2})$.
For general $r>0$, let us write
\begin{equation}
g_0(z,s,r)=g_0(z,s,0)+r\log((\sqrt2+1)z)
\end{equation}
and observe that $-\Re(r\log((\sqrt2+1)z))$ is steep descent for the contour $1-Re^{-i\phi}$, hence also $-\Re(g_0(z,s,r))$ is steep descent along the same contour.

By localizing the integral \eqref{pimprovedrepr}, we can write
\begin{multline}\label{pimprovedlocalize}
\big|P_l^{(n)}(v+j)\big|\\=2^{-5/6}n^{1/3}e^{\Re(-ng_0(1-R,s,r))}\bigg(\bigg|\frac1{2\pi i}\int_{-\delta}^\delta\d\phi\,Re^{n(-g_0(1-Re^{-i\phi},s,r)+g_0(1-R,s,r))}\bigg|+\O(e^{-\wt cn})\bigg).
\end{multline}
By the fact that the difference of $-\Re(g_0(z,s,r))+\Re(g_0(z,s,0))$ was previously shown to be steep descent along the integration contour and by Taylor expansion,
\begin{equation}\begin{aligned}
\Re(-g_0(1-Re^{-i\phi},s,r)+g_0(1-R,s,r))&\le\Re(-g_0(1-Re^{-i\phi},s,0)+g_0(1-R,s,0))\\
&=-\wt\gamma\phi^2+\O(\phi^4)
\end{aligned}\end{equation}
where $\wt\gamma=\frac1{2\sqrt2}\frac{s(1-8s+8s^2)}{1+8s+8s^2}$.
Then the integral in absolute value on the right-hand side of \eqref{pimprovedlocalize} can be bounded in the same way as in \eqref{plocalcompute} by $C/\sqrt{n\wt\gamma}$.
This bound is the largest when $s$ is small in which case together with the $n^{1/3}$ prefactor it is of order $1/\sqrt{n^{1/3}s}=C/\sqrt t$.

If $s<\varepsilon$, then Taylor expansion yields
\begin{equation}\begin{aligned}
-n\Re(g_0(1-R,s,0))&=-n\frac{32\sqrt2}3s^3(1+\O(\varepsilon))=-\frac43t^3(1+\O(\varepsilon)),\\
-n\Re(r\log((\sqrt2+1)(1-R)))&=-8nrs(1+\O(\varepsilon))=-2(y+\zeta)t(1+\O(\varepsilon))
\end{aligned}\end{equation}
from which it follows
\begin{equation}\label{pfors<eps}
\big|P_l^{(n)}(v+j)\big|\le\frac C{\sqrt t}e^{-\frac43t^3-2(y+\zeta)t}.
\end{equation}

For $s\in(\varepsilon,\frac12-\frac1{2\sqrt2})$, there is a $\delta=\delta(\varepsilon)>0$ such that
\begin{equation}\label{concaveineqs}
-\Re(g_0(1-R,s,0))\le-2^{7/6}\delta s,\qquad-\Re(r\log((\sqrt2+1)(1-R)))\le-\delta rs.
\end{equation}
To prove the first inequality in \eqref{concaveineqs}, first remark that $g_0(1-R(s=0),0,0)=0$, $\frac{\d}{\d s}g_0(1-R(s),s,0)|_{s=0}=0$, and
\begin{equation}
\frac{\d^2}{\d s^2}g_0(1-R(s),s,0)=\frac{128}{\sqrt2}\frac s{1-48s^2+64s^4}
\end{equation}
which is positive for $s\in(0,\frac12-\frac1{2\sqrt2})$, hence $\Re(g_0(1-R,s,0))$ is a convex function of $s$.
For the second bound in \eqref{concaveineqs}, it is enough to take the first derivative.
Then using \eqref{pimprovedlocalize} and \eqref{concaveineqs}, one gets a bound
\begin{equation}\label{pfors>eps}
\big|P_l^{(n)}(v+j)\big|\le Ce^{-n2^{7/6}\delta s-n\delta rs}=Ce^{-\delta n^{2/3}t-4\delta(y+\zeta)t}.
\end{equation}

To finish the proof of \eqref{improvedbound}, let $K$ be a large fixed constant.
If $t\le K$, then let us use Proposition~\ref{prop:pbound} to see that the integral on the left-hand side of \eqref{improvedbound} is at most $Ce^{-c(y+\zeta)}=C'e^{-c(K+y+\zeta)}\le C'e^{-c(t+y+\zeta)}$.
If $t>K$, then both \eqref{pfors<eps} and \eqref{pfors>eps} give a bound $Ce^{-ct-c(y+\zeta)t}\le Ce^{-ct-cK(y+\zeta)}$ proving \eqref{improvedbound}.
\end{proof}

\begin{proof}[Proof of Proposition~\ref{prop:Hconv}]
Suppose first that the function $g$ is continuous on $[L,M]$ with square integrable derivative.
In this case, the hitting position is a function of the hitting time and $g$ by the continuity hence it is enough to prove the convergence of the hitting times.

It follows from Donsker's invariance principle that the rescaled random walk trajectories $(2^{5/6}n^{-1/3}S_{b_n(t)/2})_{t\in[L,M]}$
converge weakly on the space of continuous functions on $[L,M]$ with the uniform topology to the trajectory of the Brownian motion $(b(t))_{t\in[L,M]}$ with diffusion coefficient $2$.
By the Portmanteau theorem the weak convergence implies that for any $s\in[L,M]$
\begin{equation}\label{portmanteau}
\P\left(2^{5/6}n^{-1/3}S_{\frac{b_n(t)}2}\le g(t)-t^2\mbox{ for }t\in[L,s]\right)\to\P\left(b(t)\le g(t)-t^2\mbox{ for }t\in[L,s]\right)
\end{equation}
as $n\to\infty$ provided that the event $E_s=\{b(t)\le g(t)-t^2\mbox{ for }t\in[L,s]\}$ on the right-hand side of \eqref{portmanteau} is a continuity set for the Brownian motion measure,
i.e.\ for its boundary $\P(\partial E_s)=0$.
If it is the case, then the weak convergence of hitting times \eqref{Hconv} follows because \eqref{portmanteau} is equivalent to the convergence of the tail probabilities
\begin{equation}
\P\Big(2^{7/6}n^{-2/3}\wh T^{u,\frac{b_n(L)}2}_+>s\Big)\to\P\Big(T^{\xi,L}_+>s\Big).
\end{equation}

What remains to prove is that $\P(\partial E_s)=0$ for any $s\in[L,M]$.
Since the derivative of $g(t)-t^2$ is square integrable on $[L,M]$, it satisfies Novikov's condition and
by the Cameron-Martin theorem $b(t)-g(t)+t^2$ is a Brownian motion on $t\in[L,M]$ under an equivalent measure, hence $\P(\sup_{t\in[L,s]}(b(t)-g(t)+t^2)=0)=0$.

Suppose that $g(t)-t^2$ has a jump downwards at $s_1$.
To prove the joint convergence in \eqref{Hconv}, we have to see that the boundary of the event $E_{s_1}\cap\{b(s_1)\in I\}$ for any interval $I$ has $0$ measure under the law of the Brownian motion.
The boundary under the topology induced by the uniform distance is a subset of the union of $\partial E_{s_1}$ and the event that $b(s_1)$ is equal to one of the endpoints of $I$ which both have $0$ measure.
If $g(t)-t^2$ has finitely many jumps and finitely many intervals where its derivative is square integrable, then the combination of the above arguments and induction leads to the proof.
\end{proof}

\begin{proof}[Proof of Lemma~\ref{lemma:intkerneldecay}]
In order to prove \eqref{kdecay}, we use Propositions~\ref{prop:qbound} and~\ref{prop:improvedbound} to bound the $q$ and $p$ factors in \eqref{defk} respectively.
We separate two regimes where the hitting time probability is bounded differently.
Let $\ul g=\min_{\tau\in[L,M]}(g(\tau)-\tau^2)$.
Note that the hitting position is lower bounded by $\frac n{\sqrt2}+2^{-5/6}\ul gn^{1/3}$, hence after rescaling, $\zeta\ge\ul g$ holds.
In the first regime where the starting position $u$ of the random walk is at least $\frac n{\sqrt2}+2^{-5/6}\ul gn^{1/3}$ which corresponds to $\xi\ge\ul g$, we simply use that
\begin{equation}\label{hittingprobbound}
\sum_{l=\frac{b_n(T)}2}^{\frac{b_n(T+1)}2}\sum_{v\in\Z}\P\Big(\wh T^{u,\frac{b_n(L)}2}_+=l,\wh X^{u,m}_+=v\Big)\le1.
\end{equation}
Then by using \eqref{qbound}, \eqref{improvedbound} with $e^{-c\zeta}\le e^{-c\ul g}$ on the right-hand side replaced by a constant and \eqref{hittingprobbound}, we get
\begin{equation}\label{highstart}\begin{aligned}
&\sum_{u\ge\frac n{\sqrt2}+2^{-5/6}\ul gn^{1/3}}\sum_{l=\frac{b_n(T)}2}^{\frac{b_n(T+1)}2}\sum_{v\in\Z}\Big|2^{-5/6}n^{1/3}\P\Big(\wh T^{u,\frac{b_n(L)}2}_+=l,\wh X^{u,m}_+=v\Big)k_n^{u,l,v}(i,j)\Big|\\
&\qquad\le\sum_{u\ge\frac n{\sqrt2}+2^{-5/6}\ul gn^{1/3}}C^2n^{-1/3}e^{-c(x+\xi+T+y)}\sum_{l=\frac{b_n(T)}2}^{\frac{b_n(T+1)}2}\sum_{v\in\Z}\P\Big(\wh T^{u,\frac{b_n(L)}2}_+=l,\wh X^{u,m}_+=v\Big)\\
&\qquad=C'\int_{\xi\ge\ul g}\d\xi\,e^{-c(x+\xi+T+y)}=\frac{C'e^{-c\ul g}}ce^{-c(x+T+y)}.
\end{aligned}\end{equation}

The second regime is where the random walk starts below $\frac n{\sqrt2}+2^{-5/6}\ul gn^{1/3}$, i.e.\ $\xi<\ul g$.
For these values we apply the large deviation bound of Proposition~\ref{prop:ldp} as follows.
Let us decompose the probability that the hitting happens between $b_n(T)/2$ and $b_n(T+1)/2$ according to the value of the random walk at $b_n(T)$ by writing
\begin{equation}\label{hittingsumdecomp}\begin{aligned}
&\sum_{l=\frac{b_n(T)}2}^{\frac{b_n(T+1)}2}\sum_{v\in\Z}\P\Big(\wh T^{u,\frac{b_n(L)}2}_+=l,\wh X^{u,m}_+=v\Big)\\
&\quad\le\P\Big(S_{2^{-7/6}(T-L)n^{2/3}}\ge\frac n{\sqrt2}+2^{-5/6}\ul gn^{1/3}-u\Big)\\
&\qquad+\int_\xi^{\ul g}\d\eta\,\P\Big(S_{2^{-7/6}(T-L)n^{2/3}}\ge\frac n{\sqrt2}+2^{-5/6}\eta n^{1/3}-u\Big)\\
&\qquad\quad\times\P\bigg(\sup_{0\le k\le2^{-7/6}n^{2/3}}S_k>\frac n{\sqrt2}+2^{-5/6}(\ul g-\eta)n^{1/3}\bigg)\\
&\qquad+\P\left(S_{2^{-7/6}(T-L)n^{2/3}}<0\right)\P\bigg(\sup_{0\le k\le2^{-7/6}n^{2/3}}S_k>\frac n{\sqrt2}+2^{-5/6}\ul gn^{1/3}-u\bigg).
\end{aligned}\end{equation}
We bound the first term on the right-hand side of \eqref{hittingsumdecomp} using Proposition~\ref{prop:ldp} with $m=2^{-7/6}(T-L)n^{2/3}$ and $x=2^{1/3}\frac{\ul g-\xi}{T-L}n^{-1/3}$ as
\begin{equation}\label{ldpinuse}\begin{aligned}
\P\Big(S_{2^{-7/6}(T-L)n^{2/3}}\ge\frac n{\sqrt2}+2^{-5/6}\ul gn^{1/3}-u\Big)
&\leq e^{-\varepsilon2^{-7/6}(T-L)n^{2/3}\left(2^{1/3}\frac{\ul g-\xi}{T-L}n^{-1/3}\right)^2}\\
&=e^{-\varepsilon2^{-1/2}\frac{(\ul g-\xi)^2}{T-L}}.
\end{aligned}\end{equation}
Very similarly by Proposition~\ref{prop:ldp} the integral on the right-hand side of \eqref{hittingsumdecomp} is upper bounded by
\begin{equation}\label{ldpforintegral}\begin{aligned}
&\int_\xi^{\ul g}\d\eta\,\P\Big(S_{2^{-7/6}(T-L)n^{2/3}}\ge\frac n{\sqrt2}+2^{-5/6}\eta n^{1/3}-u\Big)\\
&\quad\times\P\bigg(\sup_{0\le k\le2^{-7/6}n^{2/3}}S_k>\frac n{\sqrt2}+2^{-5/6}(\ul g-\eta)n^{1/3}\bigg)\\
&\qquad\le\int_\xi^{\ul g}\d\eta e^{-\varepsilon2^{-1/2}\frac{(\eta-\xi)^2}{T-L}-\varepsilon2^{-1/2}(\ul g-\eta)^2}
\le\frac{\sqrt\pi}{\sqrt{\varepsilon2^{-1/2}(1+\frac1{T-L})}}e^{-\varepsilon2^{-1/2}\frac{(\ul g-\xi)^2}{1+T-L}},
\end{aligned}\end{equation}
where the last inequality follows by computing the Gaussian integral in $\eta$ over $\R$.
Note that the prefactor in front of the exponential on the right-hand side of \eqref{ldpforintegral} is bounded by a constant for any $T>L$.
Further, the last term on the right-hand side of \eqref{hittingsumdecomp} is at most $1\cdot e^{-\varepsilon2^{-1/2}(\ul g-\xi)^2}$.
As a conclusion, by comparing the right-hand sides of \eqref{ldpinuse} and \eqref{ldpforintegral},
the sum of hitting probabilities on the left-hand side of \eqref{hittingsumdecomp} is at most $Ce^{-\varepsilon'(\ul g-\xi)^2/(1+T-L)}$ with some $C\in\R$ and $\varepsilon'>0$.

Hence in the $\xi<\ul g$ regime,
\begin{multline}\label{lowstart}
\sum_{u<\frac n{\sqrt2}+2^{-5/6}\ul gn^{1/3}}\sum_{l=\frac{b_n(T)}2}^{\frac{b_n(T+1)}2}\sum_{v\in\Z}\Big|2^{-5/6}n^{1/3}\P\Big(\wh T^{u,\frac{b_n(L)}2}_+=l,\wh X^{u,m}_+=v\Big)k_n^{u,l,v}(i,j)\Big|\\
\le C\int_{\xi<\ul g}\d\xi\,e^{-\varepsilon'\frac{(\ul g-\xi)^2}{1+T-L}-c(x+\xi+T+y)}.
\end{multline}
The $\xi$ dependent part of the integral above can be upper bounded by the integral over $\R$ which is a Gaussian integral
\begin{equation}\label{lowstartGauss}
\int_\R\d\xi e^{-\varepsilon'\frac{(\ul g-\xi)^2}{1+T-L}-c\xi}=\sqrt{\frac{\pi(1+T-L)}{\varepsilon'}}e^{-c\ul g+\frac{c^2(1+T-L)}{4\varepsilon'}}.
\end{equation}
Since the exponent of $t$ in \eqref{improvedbound} is arbitrary, the part of the sum in \eqref{lowstart} can still be bounded by $Ce^{-c(x+y+T)}$ and \eqref{kdecay} follows.

The proof of \eqref{kappadecay} is similar, hence we omit the fine details.
If $T>L$ and $\zeta\ge\ul g$, then there are $c>0$ and $C\in\R$ so that
\begin{equation}
\big|\Ai^{(-t)}(\zeta+y)\big|\le Ce^{-c(y+T)},\qquad\big|\Ai^{(L)}(x+\xi)\big|\le Ce^{-c(x+\xi)}
\end{equation}
for $t\in[T,T+1]$, $x,y\ge0$ and $\xi\in\R$.
Therefore, one can bound
\begin{multline}\label{hittingAi}
\int_\R\d\xi\int_T^{T+1}\int_\R\P(T^{\xi,L}_+\in\d t,X^{\xi,L}_+\in\d\zeta)\left|\Ai^{(L)}(x+\xi)\Ai^{(-t)}(\zeta+y)\right|\\
\le C^2\int_\R\d\xi e^{-c(x+\xi+T+y)}\int_T^{T+1}\int_\R\P(T^{\xi,L}_+\in\d t,X^{\xi,L}_+\in\d\zeta).
\end{multline}
The last double integral in $t$ and $\zeta$ is equal to the probability $\P(T^{\xi,L}_+\in[T,T+1])$.
Then the same steps apply as in the proof of \eqref{kdecay}: one separates the two regimes $\xi\ge\ul g$ and $\xi<\ul g$.
In the first regime, one bounds the probability on the right-hand side of \eqref{hittingAi} by $1$ and using a large deviation bound analogous to Proposition~\ref{prop:ldp} in the second.
For the latter bound, we observe that by the reflection principle,
\begin{equation}
\P\Big(\sup_{0\le s\le t}B(s)\ge xt\Big)=2\P(B(t)\ge xt)=2\big(1-\Phi(x\sqrt t)\big)
\end{equation}
which can be bounded by $e^{-tx^2/2}$ and \eqref{kappadecay} can be proved in the same way as \eqref{kdecay}.
\end{proof}

\end{document}